\documentclass[review,hidelinks,onefignum,onetabnum]{siamart220329}
\usepackage{listings}%
\usepackage{amsmath,amssymb}
\usepackage{enumitem}
\usepackage{bm}
\usepackage{cite}
\usepackage{float}

\newcommand{\p}{\partial}

\newcommand{\ds}{\displaystyle}

\usepackage{graphicx}
\usepackage{subfig}  
\usepackage{color}
\usepackage{lineno}

\newtheorem{remark}{Remark}[section]

\headers{Multiple solution computation of elliptic problems}{Lei Zhang, Xiangcheng Zheng, Shangqin Zhu}

\title{Numerical analysis of spatiotemporal high-index saddle dynamics for finding multiple solutions of semilinear elliptic problems \thanks{Submitted to... \funding{This work was partially supported by the National Natural Science Foundation of China (No. 12225102, T2321001, 12288101, 12301555), the Natural Science Foundation of Shandong Province (ZR2025QB01), the Taishan Scholars Program of Shandong Province (No. tsqn202306083), the National Key R\&D Program of China (No. 2023YFA1008903).}}
}

\author{Lei Zhang\thanks{Beijing International Center for Mathematical Research, Center for Machine Learning Research, Center for Quantitative Biology, Peking University, Beijing 100871, China,
  (\email{zhangl@math.pku.edu.cn}).}
\and Xiangcheng Zheng\thanks{School of Mathematics, State Key Laboratory of Cryptography and Digital Economy Security, Shandong University, Jinan, 250100, China, 
  (\email{xzheng@sdu.edu.cn}).}
\and Shangqin Zhu\thanks{School of Mathematics, Shandong University, Jinan, 250100, China, 
  (\email{sqzhu@mail.sdu.edu.cn}).}}

\usepackage{amsopn}

\begin{document}
\nolinenumbers
\maketitle

\begin{abstract}
This paper presents a rigorous numerical framework for computing multiple solutions of semilinear elliptic problems by spatiotemporal high-index saddle dynamics (HiSD), which extends the traditional HiSD to the continuous-in-space setting, explicitly incorporating spatial differential operators.
To enforce the Stiefel manifold constraint without introducing the analytical complications of retraction-based updates, we design a fully discrete retraction-free orthonormality-preserving scheme for spatiotemporal  HiSD. This scheme exhibits favorable structural properties that substantially reduce the difficulties arising from coupling and gradient nonlinearities in spatiotemporal HiSD. 
Exploiting these properties, we establish gradient stability and error estimates, which consequently ensure the preservation of the Morse index for the computed saddle points. The framework is further extended to the semilinear advection-reaction-diffusion equation. 
Numerical experiments demonstrate the efficiency of the proposed method in finding multiple solutions and constructing the solution landscape of semilinear elliptic problems.
To the best of our knowledge, this work presents the first rigorous full space--time accuracy analysis of the HiSD system. It reveals intrinsic connections between saddle-search algorithms and numerical methods for PDEs, enhancing their mutual compatibility for a broad range of problems.
\end{abstract}



\section{Introduction}
\subsection{Background and challenges}
Computing multiple solutions of nonlinear partial differential equations (PDEs) is critical for elucidating multiple stationary states and transition pathways, with broad applications in scientific and engineering fields, cf. the review \cite{Zho}. A typical multiple solution problem is the semilinear elliptic equation given by
\begin{equation}\label{elliptic}
	 \nabla\cdot(\mathbf{a}(x)\nabla u(x))+ f(u(x))=0,~x\in \Omega;~~  u=0,~x\in \p\Omega.
\end{equation}
Here, $\Omega$ is a bounded convex  polygonal domain in $\mathbb{R}^d$ ($1\leq d\leq 3$) with boundary $\partial \Omega$, $f(\cdot)$ is a given function and $\mathbf{a}(x)\in\mathbb R^{d\times d}$ is a coefficient matrix.
Various sophisticated algorithms have been developed to locate multiple solutions, such as the surface-walking methods \cite{EZho,Lev,Qua,ZhaDu}, the deflation technique \cite{Farr}, the iterative minimization formulation \cite{Gao}, the homotopy continuation method \cite{Hao2, Mehta}, the minimax-type methods \cite{LiuXie,Xie}, etc.
In particular, the high-index saddle dynamics (HiSD) method \cite{YinSISC}  is an effective method for finding multiple solutions and a natural choice for constructing the solution landscape, which has been applied in various fields such as studying defect configurations in liquid crystals \cite{wang2021modeling}, morphologies of confined diblock copolymers \cite{Xu}, the nucleation of quasicrystals \cite{Yin2020nucleation}, and excited states in rotational Bose--Einstein condensates \cite{YinInno}.

Despite successful applications, the standard HiSD is inherently formulated as a dynamical system designed for finite-dimensional problems. When applying HiSD to infinite-dimensional problems such as (\ref{elliptic}), spatial discretization is typically applied first to reduce the system to a finite-dimensional approximation, which is then evolved using HiSD. Consequently, the discrete spatial differential operators become embedded within HiSD. This suggests that the dynamical system formulation of HiSD should be adapted to the continuous-in-space setting to rigorously accommodate infinite-dimensional problems such as PDEs. 

Addressing this issue, a recent study \cite{ZhaZheZhu} proposed a continuous-in-space saddle dynamics for locating index-1 saddle solutions (transition states) of (\ref{elliptic}), formulated as a coupled nonlinear parabolic system. Recall that the Morse index of a non-degenerate saddle point is defined by the number of eigenvalues of the Hessian at this point with negative real parts. The numerical computation for transition states for (\ref{elliptic}) can thus be realized by performing the spatiotemporal discretization on this parabolic system. This approach significantly improves the compatibility between PDE numerical methods and multiple solution computations by HiSD, thereby facilitating the establishment of a rigorous numerical theory for fully spatiotemporal discretization.

\subsection{Spatiotemporal HiSD}
In this paper, we extend the study in \cite{ZhaZheZhu} to high-index cases by proposing the following {\it spatiotemporal HiSD} for locating index-$k$ saddle points of (\ref{elliptic}): 
\begin{equation}\label{SD}
\left\{\begin{aligned}
 u_t&=\beta\Big( \nabla\cdot(\mathbf{a}\nabla u)+ f(u)- 2\sum_{i=1}^k v_i(v_i,\nabla\cdot(\mathbf{a}\nabla u)+ f(u))\Big),\\[-0.03in]
 (v_i)_t& =\gamma\Big(\nabla\cdot(\mathbf{a}\nabla v_i)+ f'(u)v_i -\sum_{j=1}^{k}v_j(v_j, \nabla\cdot(\mathbf{a}\nabla v_i)+ f'(u)v_i )\Big),
    \end{aligned}
    \right.
\end{equation}
for $1\leq i\leq k$ and for $x\in\Omega$ and $t>0$, equipped with initial values $u_0$ and $\{v_{i,0}\}_{i=1}^k$ and zero boundary conditions.

 Here, $\{v_i(x,t)\}_{i=1}^k$ denote auxiliary variables that support the evolution of the state variable $u(x,t)$,  $\beta ,\gamma >0$ are relaxation parameters, and $(\cdot,\cdot)$ represents the inner product of $L^2(\Omega)$. The formulation of (\ref{SD}) can be viewed as the continuous-in-space version of the saddle dynamics proposed in \cite{quapp1}. Notably, when $k=1$, (\ref{SD}) reduces to the index-1 saddle dynamics analyzed in \cite{ZhaZheZhu}.
  
A fundamental property of (\ref{SD}) is the preservation of orthonormality: if the initial set $\{v_{i,0}\}_{i=1}^k$ is orthonormal, then $\{v_{i}\}_{i=1}^k$ remains orthonormal for all $t\geq 0$. This property inherently dictates the  dynamics of $u$ and is critical for convergence to the target saddle points. To demonstrate the impact of orthonormality, we compare the convergence behavior of the proposed orthonormality-preserving scheme (\ref{app}) with a non-orthonormal variant, obtained by removing the correction terms $(v_{i, h}^{n-\frac{1}{2}}, v_{i, h}^{n-\frac{1}{2}})$  from (\ref{app}). We set $k = 3$ in both schemes to locate an index-3 saddle point, using the initial values $u_0=\sin x$ and $v_{i,0}=\sqrt{2/\pi}\sin (ix)$ ($1\leq i\leq 3 $). All other parameters follow Case (a) of Example 1 in Section \ref{sec7}. Numerical solutions of $u$ at different time instants are shown in Figure \ref{NPO}. The results indicate that the orthonormality-preserving scheme (\ref{app}) successfully converges to an index-3 saddle point, whereas the non-orthonormal scheme deviates to an index-2 saddle point. This example suggests the importance of preserving orthonormality within the numerical scheme.
 
 \vspace{-0.123in}
 
  \begin{figure}[h]
\centering 
\includegraphics[width=4.5in,height=1.5in]{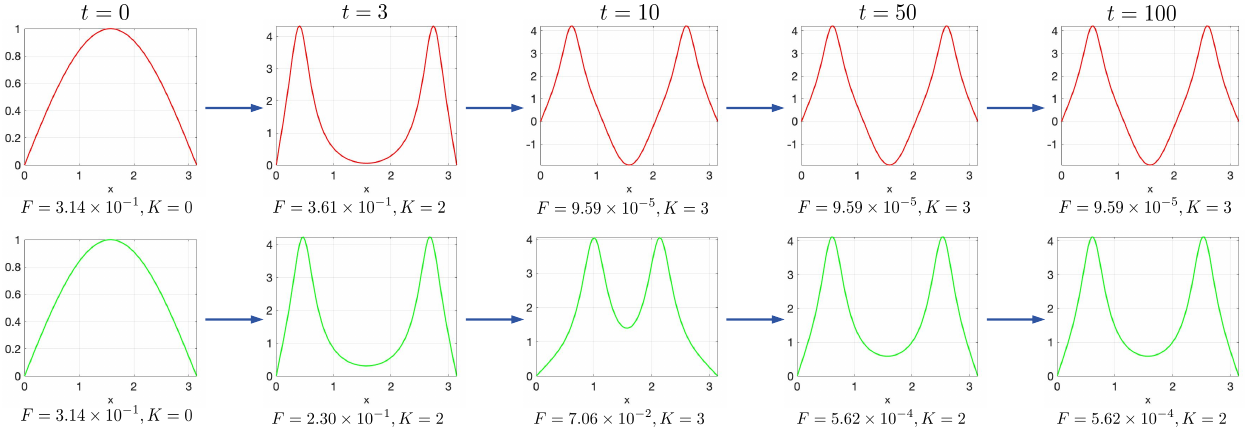} 
\caption{Numerical solutions of $u$ at different time instants for the orthonormality-preserving scheme (\ref{app}) (red) and its non-orthonormal variant (green). Here, $F:= \|F(u_h^n)\|_{l^\infty}$ (cf. Section \ref{sec7}) measures the degree of convergence  and $K$ denotes the number of negative eigenvalues of the Hessian at each state.}
\label{NPO} 
\end{figure}

\vspace{-0.401in}

\subsection{Contributions of the work} 
Existing analyses of HiSD schemes, such as \cite{Z3} and subsequent works, primarily focus on retraction schemes based on the Gram--Schmidt procedure. However, these studies typically rely on the Lipschitz condition of the problem, which is typically not applicable to PDEs. 
Recently, \cite{ZhaZheZhu} analyzed a retraction scheme for (\ref{SD}) in the index-1 case ($k=1$). Extending this analysis to high-index cases ($k>1$) presents significant challenges. The primary obstacle is that for $k>1$, the retraction step must enforce mutual orthogonality among auxiliary variables, introducing complex couplings between the retraction operator and the gradients of numerical solutions. Consequently, deriving gradient bounds and error estimates becomes exceedingly intricate. Moreover, the analysis in \cite{ZhaZheZhu} relies on an $L^\infty$ boundedness assumption for the numerical solution of $u$ to handle coupling and gradient nonlinearities; this is considered a restrictive assumption in standard numerical analysis.

To address these challenges, a retraction-free orthonormality-preserving scheme is proposed for the spatiotemporal HiSD (\ref{SD}).  While orthonormality is proved by induction, the coupled dynamics of $v_i$ in (\ref{SD}) do not constitute a gradient flow on the Stiefel manifold. Consequently, gradient bounds for $v_i$ cannot be derived from standard energy decay arguments. Fortunately, we find that the proposed scheme exhibits favorable structural properties that can significantly reduce the difficulties in handling the coupling and the gradient nonlinearity. Specifically, we employ the nonlinear structures of the scheme to establish gradient stability and error estimates, without imposing the $L^\infty$ boundedness of the numerical solutions required in \cite{ZhaZheZhu}.
 We further extend this framework to semilinear advection-reaction-diffusion equation. Our results validate the accuracy of multiple solution computations and the index-preservation property, providing a solid theoretical foundation for solution landscape investigations.

The rest of the paper is organized as follows. We introduce preliminary results in Section 2, and then develop the retraction-free orthonormality-preserving  scheme in Section 3. We bound gradients of the numerical solutions in Section 4, which is critical for error analysis in Section 5. The extension to semilinear advection-reaction-diffusion equation is considered in Section 6. Numerical experiments are performed in Section 7, and we address concluding remarks in the last section.

 \section{Preliminaries}
 Let $L^p(\Omega)$ and $W^{m,p}(\Omega)$ for $0\leq m\in\mathbb N$ and $1\leq p\leq \infty$ be standard Sobolev spaces, equipped with standard norms \cite{Ada}. In particular, we set $H^m(\Omega):=W^{m,2}(\Omega)$ and $H^m_0(\Omega)$ denotes the closure of $C^\infty_0(\Omega)$, the space of infinitely differentiable functions with compact support in $\Omega$,  with respect to the norm $\|\cdot\|_{H^m(\Omega)}$. For a Banach space $\mathcal X$ and some $T>0$, the Bochner space $L^p(0,T;\mathcal X)$ contain functions $g$ that are finite under the norm $\|g\|_{L^p(0,T;\mathcal X)}:=\|\|g\|_{\mathcal X}\|_{L^p(0,T)}$ \cite{Ada}.
Then the space $W^{m,p}(0, T;\mathcal X)$ contains functions that are finite under the  norm $\|u\|_{W^{m,p}(0, T;\mathcal X)}:= \sum_{k=0}^m  \|u^{(k)}_t\|_{L^p(0,T;\mathcal X)}$ \cite{Ada}. For simplicity, we denote $\|\cdot\|:=\|\cdot\|_{L^2(\Omega)}$ and $(g,\tilde g):=\int_\Omega g(x)\tilde g(x)dx$, and omit $\Omega$ in notations of spatial norms. Furthermore, we omit the space variable $x$ in functions, for instance, we denote $u(x,t)$ by $u(t)$. Finally, for a vector-valued function $\mathbf{w}(x)\in\mathbb R^d$, we define $\|\mathbf{w}\|:=\|\|\mathbf{w}\|_{l^2}\|$ where $l^2$ refers to the $l^2$-norm of a vector or a matrix.

In this work, we make the following standard assumptions on the coefficient matrix $\mathbf{a}$ and the nonlinear term $f$:

\textit{Assumption A}: The coefficient matrix $\mathbf{a}$ is symmetric, bounded and uniformly elliptic,  which means that $\mathbf{a}$ satisfies $\mathbf{a}=\mathbf{a}^\top$, each entry of $\mathbf{a}$ belongs to $W^{1,\infty}$ and there exists a constant $a_0>0$ such that $\mathbf{w}^\top \mathbf a\mathbf w\geq a_0|\mathbf w|^2$ for any $\mathbf w\in\mathbb R^d$. The $f$ and $f'$ are local Lipschitz functions on $\mathbb R$, that is, for any $r>0$ there exists a constant $L_r>0$ depending on $r$ such that 
$
|f(z_1)-f(z_2)|+|f'(z_1)-f'(z_2)|\leq L_r|z_1-z_2|$ for any $ -r\leq z_1, z_2\leq r$.

\begin{remark}
Note that the conditions on $\mathbf a$ are commonly-used assumptions in the literature of elliptic problems and, since $\mathbf a$ is symmetric and positive definite by assumption,  there exists a symmetric and positive definite matrix  $\sqrt{\mathbf{a}}$ such that $(\sqrt{\mathbf{a}})^2=\mathbf a$. The inverse of $\mathbf{a}$ denoted by $\mathbf{a}^{-1}$ is also symmetric and positive definite such that there exists a symmetric and positive definite matrix  $\sqrt{\mathbf{a}^{-1}}$ satisfying $(\sqrt{\mathbf{a}^{-1}})^2=\mathbf{a}^{-1}$. As the coercive coefficient of $\mathbf{a}$ is $a_0$, the eigenvalues of $\sqrt{\mathbf{a}^{-1}}$ have an upper bound $\frac{1}{\sqrt{a_0}}$ for any $x\in\Omega$ such that 
\begin{equation}\label{zjy15}
\|\sqrt{\mathbf{a}^{-1}}\mathbf{w}\|=\|\|\sqrt{\mathbf{a}^{-1}}\mathbf{w}\|_{l^2}\|  \leq \|\|\sqrt{\mathbf{a}^{-1}}\|_{l^2}\|\mathbf{w}\|_{l^2}\|\leq \frac{1}{\sqrt{a_0}}\|\mathbf w\|.
\end{equation}
\end{remark}

We use $Q$ to denote a generic positive constant that may assume different values at different occurrences.
All these constants may depend on $T$ and certain norms of solutions but are independent from parameters of numerical methods such as the time step size, the number of steps and the spatial mesh size. 

For a time-dependent matrix $\mathbf{A}(t)\in \mathbb{R}^{k\times k} $ with elements $\mathbf{A}_{ij}=a_{ij}(t)$, we define a norm as 
$
 	\|\mathbf{A}\|=\sum_{i,j=1}^k |a_{ij}|,
$
 which, by direct calculations, satisfies
$
 	\|\mathbf{A}\mathbf{B}\|\leq 	\|\mathbf{A}\| \|\mathbf{B}\|$ and $ \|\int _0^t \mathbf{A}(s) ds \| \leq \int _0^t \|\mathbf{A}(s) \|ds$ for any $\mathbf{A}, \mathbf{B}\in \mathbb{R}^{k\times k}$ and $t\geq 0$.

\begin{theorem}\label{orth}
If the initial values $\{v_{i,0}\}_{i=1}^k$ of (\ref{SD}) satisfy $(v_{i,0},v_{j,0})=\delta_{ij}$ for $ 1\leq i,j \leq k$, then $(v_{i},v_{j})=\delta_{ij}$ for $1\leq i,j \leq k$ and for any $t>0$.
\end{theorem}
\begin{proof}
Define  symmetric matrices $\mathbf{M}(t)\in\mathbb{R}^{k\times k}$  with elements $m_{ij}:=(v_i,\nabla\cdot(\mathbf{a}\nabla v_j)+f'(u)v_j)$ for $1\leq i,j \leq k$ and $\mathbf{H}(t)\in\mathbb{R}^{k\times k}$ with elements $ h_{ij}:=(v_i,v_j)-\delta_{ij}$ for $1\leq i,j \leq k$. According to (\ref{SD}), we have
\begin{equation*}
\begin{aligned}
&((v_i,v_i)-1)_t=2(v_{i,t},v_i)=2\gamma \Big[ (\nabla\cdot(\mathbf{a}\nabla v_i)+ f'(u)v_i, v_i)\\
&~~-\sum_{l=1}^{k}(v_l, v_i)(v_l, \nabla\cdot(\mathbf{a}\nabla v_i)+ f'(u)v_i) \Big] =-2\gamma \sum_{l=1}^{k}m_{il}h_{li}=-\gamma (\mathbf{M}\mathbf{H}+\mathbf{H}\mathbf{M})_{ii},
\end{aligned}
\end{equation*}
and, similarly, 
$(v_i, v_j)_t=(v_{i, t},v_j)+ (v_i, v_{j,t})=-\gamma (\mathbf{M}\mathbf{H}+\mathbf{H}\mathbf{M})_{ij}$.
Thus we obtain a matrix equation
$
	\frac{d\mathbf{H}(t)}{dt}+\gamma (\mathbf{M}(t)\mathbf{H}(t)+\mathbf{H}(t)\mathbf{M}(t))=0$ with
 $\mathbf{H}(0)=0$. Integrate both sides of this equation  with respect to $t$ and apply the matrix norm $\|\cdot\|$ to obtain
$
\|\mathbf{H}(t)\|\leq 2\gamma \int_0^t \|\mathbf{M}(s)\|\cdot \|\mathbf{H}(s)\| ds.
$
An application of the Gronwall inequality  completes the proof.
\end{proof}

\section{Scheme and orthonormality preservation}
We propose a retraction-free orthonormality-preserving scheme for (\ref{SD}) on a finite time interval $[0,T]$. It will be shown in Corollary \ref{cor1} that the numerical solution at $T$ could approximate the target saddle point solution with an arbitrarily small error for $T$ large enough, which justifies the sufficiency of restricting $t\in [0,T]$.
 
\subsection{Numerical scheme}
Define a quasi-uniform partition of $\Omega$ with mesh size $h$, and let $S_h$
be the space of continuous and piecewise linear functions on $\Omega$ with respect to the
partition. The elliptic projection operator $P:  H^1_0\rightarrow S_{h}$ defined by
$\big(\mathbf{a}\nabla(g- P g), \nabla \chi \big)=0$ for any $\chi\in S_h$
  satisfies the following estimate \cite{Bre5}
\begin{equation}\label{eta1}
	\|g- P g\|+h\|\nabla(g-P g)\|\leq Qh^2\|g\|_{H^2} \text{ for any } g\in H^2\cap H^1_0. 
\end{equation}
To obtain the weak formulation of the problem, we compute the inner product of the equations in (\ref{SD}) with test functions $\chi_1, \chi_2 \in H^1_0$, respectively, to get
\begin{equation}\label{weak}
\left\{\begin{aligned}
&\beta^{-1}(u_{t}, \chi_1) + (\mathbf{a}\nabla u, \nabla \chi_1) - 2\sum_{i=1}^k(\mathbf{a}\nabla v_i,  \nabla u)(v_i,\chi_1) \\[-0.1in]
&\qquad\qquad\qquad\qquad= (f(u),\chi_1)-2\sum_{i=1}^k(f(u),v_i)(v_i,\chi_1), \\
&\gamma^{-1}((v_i)_t,\chi_2)= -(\mathbf{a}\nabla v_i, \nabla \chi_2)+(f'(u)v_i,\chi_2)\\
&\qquad\qquad\qquad\qquad+ \sum_{j=1}^k(v_j,\chi_2)\big[(\mathbf{a}\nabla v_i,\nabla v_j)- (f'(u)v_i, v_j)\big],~1\leq i\leq k.
\end{aligned}
    \right.
\end{equation}

Given $0<N\in \mathbb{N}$, we define a uniform temporal partition of the interval $[0,T]$ by $t_n=n\tau$ for $0\leq n\leq N$ with time step size $\tau = T/N$. We approximate the derivative by the backward Euler scheme at $t_n$ as $
	y_t(t_n)=\bar \partial_t y_n+R_n^y := \frac{y(t_n)-y(t_{n-1})}{\tau}+\frac{1}{\tau}\int _{t_{n-1}}^{t_n} y_{tt}(t)(t-t_{n-1})dt,
$
where $y$ refers to $u$ or $v_i$.  
We invoke such temporal discretization in (\ref{weak}) and multiply $(v_i, v_i)=1$ on several terms to get
\begin{equation}\label{ref}
\left\{
\begin{array}{l}
\!\!\!\! \ds \beta ^{-1}\big(\bar \partial _t u(t_n)+R_n^u, \chi_1 \big) +  \big(\mathbf{a}\nabla u(t_{n}), \nabla \chi_1 \big) - 2\sum_{i=1}^k\big(\mathbf{a}\nabla v_i(t_{n}), \nabla u(t_{n})\big)(v_i(t_{n}), \chi_1) \\[-0.05in]
\ds \!\!\!\quad =  (f(u(t_{n})), \chi_1) - 2\sum_{i=1}^k\big(f(u(t_{n})), v_i(t_{n})\big)(v_i(t_{n}),\chi_1),\\[0.2in]
\!\!\!\!\gamma^{-1}\big(\bar \partial _t v_i(t_n)\!+\!R_n^{v_i}, \chi_2 \big)\!=\! - (v_i(t_n), v_i(t_n))\big[(\mathbf{a}\nabla v_i(t_{n}), \nabla \chi_2)\!-\!(f'(u(t_{n}))v_i(t_{n}),\chi_2)\big]\\[0.15in]
\!\!\!\quad + (v_i(t_{n}), \chi_2)\big[(\mathbf{a}\nabla v_i(t_{n}),\nabla v_i(t_{n}))-(f'(u(t_{n}))v_i(t_{n}), v_i(t_{n}))\big]\\[0.15in]
\ds \!\!\!\quad + \sum_{l\neq i}(v_i(t_n), v_i(t_n))(v_l(t_{n}), \chi_2)\big[(\mathbf{a}\nabla v_i(t_{n}),\nabla v_l(t_{n}))-(f'(u(t_{n}))v_i(t_{n}), v_l(t_{n}))\big].
\end{array}
\right.
\end{equation}

Before giving the fully-discrete scheme, we discuss the initial values of the numerical solutions denoted by $\{u^n_h,v^n_{i,h}\}_{n=0,i=1}^{N,k}$. We could conventionally set $u_h^0=P u_0$, while, to ensure the orthonormality of the initial values $\{v_{i,h}^0\}_{i=1}^k$, we generating them by orthonormalizing $\{Pv_{i,0}\}_{i=1}^k$ using the Gram-Schmidt procedure.

Now we show that $\{v_{i,h}^0\}_{i=1}^k$ are well approximations to $\{Pv_{i,0}\}_{i=1}^k$ (and thus $\{v_{i,0}\}_{i=1}^k$).
 By
$
(Pv_{i,0},Pv_{j,0})=(v_{i,0},v_{j,0})+(Pv_{i,0}-v_{i,0},Pv_{j,0})+(v_{i,0},Pv_{j,0}-v_{j,0}),
$
 $(v_{i,0},v_{j,0})=\delta_{i,j}$ for $1\leq i,j\leq k$ and (\ref{eta1}), we have
$
|(Pv_{i,0},Pv_{j,0})-\delta_{i,j}|\leq Qh^2$ for $1\leq i,j\leq k
$
and $v_{i,0}\in H^2\cap H^1_0$. Then we follow exactly the same derivation as \cite[Lemma 4.2]{Z3} to get
\begin{align}\label{jy8}
\|Pv_{i,0}-v_{i,h}^0\|\leq Qh^2,~~1\leq i\leq k,\text{ for }h\text{ small enough}.
\end{align}
Furthermore, we invoke this and the inverse estimate $\|\chi_h\|_{H^1}\leq Qh^{-1}\|\chi_h\|$ for any $\chi_h\in S_h$ \cite{Bre5} to get the gradient bound of $v_{i,h}^0$
$$\|\nabla v_{i,h}^0\|\leq \|\nabla (v_{i,h}^0-Pv_{i,0})\|+\|\nabla Pv_{i,0}\|\leq Qh^{-1}\|v_{i,h}^0-Pv_{i,0}\|+\|\nabla v_{i,0}\|\leq Q. $$

Now we drop the truncation errors in (\ref{ref}) and adjust time steps in several terms (e.g. we replace $v_i(t_n)$ by $v_{i, h}^{n-\frac{1}{2}}:= \frac{v_{i, h}^{n}+v_{i, h}^{n-1}}{2} $ in several places for the sake of orthonormality preservation) to propose the retraction-free orthonormality-preserving scheme: find $\{u^n_h,v^n_{i,h}\}_{n=1,i=1}^{N,k}$ such that for any $\chi_1,\chi_2 \in S_h$
\begin{equation}\left\{
\begin{array}{l}
 \ds \beta ^{-1}\big(\frac{u^n_h-u^{n-1}_h}{\tau}, \chi_1\big)+(\mathbf{a}\nabla u^{n}_h, \nabla \chi_1)-2\sum_{i=1}^k(\mathbf{a}\nabla v^{n-1}_{i, h}, \nabla u^{n}_h)(v^{n-1}_{i, h}, \chi_1)\\[0.1in]
\ds \qquad  =(f(u^{n-1}_h), \chi_1)-2\sum_{i=1}^{k}(f(u^{n-1}_h), v^{n-1}_{i, h})(v^{n-1}_{i, h},\chi_1),\\[0.2in]
\ds \gamma^{-1}\big(\frac{ v^{n}_{i, h}-v^{n-1}_{i, h}}{\tau}, \chi_2\big)=-(v_{i, h}^{n-\frac{1}{2}}, v_{i, h}^{n-\frac{1}{2}})\big[ (\mathbf{a}\nabla v^{n-\frac{1}{2}}_{i, h}, \nabla \chi_2)-(f'(u_h^n)v_{i, h}^{n-\frac{1}{2}}, \chi_2) \big]\\[0.15in]
\ds \quad +(v_{i, h}^{n-\frac{1}{2}}, \chi_2)\big[(\mathbf{a}\nabla v_{i, h}^{n-\frac{1}{2}}, \nabla v_{i, h}^{n-\frac{1}{2}})-(f'(u_h^n)v_{i, h}^{n-\frac{1}{2}}, v_{i, h}^{n-\frac{1}{2}})\big]\label{app}
\\[0.15in]
\ds \quad +\sum_{l< i}(v_{i, h}^{n-\frac{1}{2}}, v_{i, h}^{n-\frac{1}{2}})(v_{l,h}^{n}, \chi_2)\big[(\mathbf{a}\nabla v_{i, h}^{n-\frac{1}{2}}, \nabla v_{l, h}^{n})-(f'(u_h^n)v_{i, h}^{n-\frac{1}{2}}, v_{l, h}^{n})\big]\\[0.15in]
\ds \quad +\sum_{l> i}(v_{i, h}^{n-\frac{1}{2}}, v_{i, h}^{n-\frac{1}{2}})(v_{l,h}^{n-1}, \chi_2)\big[(\mathbf{a}\nabla v_{i, h}^{n-\frac{1}{2}}, \nabla v_{l, h}^{n-1})-(f'(u_h^n)v_{i, h}^{n-\frac{1}{2}}, v_{l, h}^{n-1})\big].
\end{array}\right.
\end{equation}

\begin{remark}
In Section \ref{sec32}, we will find that the adjusting terms  $(v_{i, h}^{n-\frac{1}{2}}, v_{i, h}^{n-\frac{1}{2}})$ is critical to ensure the preservation of orthonormality. To understand the underlying mechanism, we refer the relation $(v_{i, h}^{n-\frac{1}{2}}, v_{i, h}^{n-\frac{1}{2}})=1-\frac{\|v^{n}_{i, h}-v^{n-1}_{i, h}\|^2}{4}$ from (\ref{Zjy1}),  which is valid after we prove the orthonormality preservation of (\ref{app}). If we replace $(v_{i, h}^{n-\frac{1}{2}}, v_{i, h}^{n-\frac{1}{2}})$ by $1-\frac{\|v^{n}_{i, h}-v^{n-1}_{i, h}\|^2}{4}$ in (\ref{app}), then the terms containing $\frac{\|v^{n}_{i, h}-v^{n-1}_{i, h}\|^2}{4}$, which are not generated from the original dynamics of $v_i$, could be viewed as the imposed Lagrangian multipliers or high-order perturbations that help to preserve the  orthonormality. \end{remark}

\subsection{Orthonormality preservation}\label{sec32}
We prove the orthonormality preservation of the scheme (\ref{app}). We first present the following auxiliary result.
\begin{lemma}\label{jump}
If  $(v_{i, h}^{n-1}, v_{j, h}^{n-1})=\delta_{ij}$ for $1\leq i,j\leq k$ holds for some $1\leq n\leq N$ and 
\begin{equation}\label{unless}
(\tau \gamma)^{-1} \neq \frac{1}{2}\big[(\mathbf{a}\nabla v_{i, h}^{n-\frac{1}{2}}, \nabla v_{i, h}^{n-\frac{1}{2}})-(f'(u_h^n)v_{i, h}^{n-\frac{1}{2}}, v_{i, h}^{n-\frac{1}{2}})\big],~~1\leq i\leq k,
\end{equation}
then $
	(v_{p, h}^{n}, v_{q, h}^{n-1})=0$ for $1\leq p < q \leq k.
$
\end{lemma}
\begin{remark}
It is worth mentioning that the condition (\ref{unless}) could be naturally satisfied for $\tau$ small enough. Specifically,
as $\tau$ tends to $0$, the left-hand side of (\ref{unless}) tends to infinity while the right-hand side may get close to its continuous counterpart, which is bounded if $\|v_i\|_{H^1(\Omega)}\leq Q$ for $t\in [0,T]$ and $|u|\leq Q$ for $(x,t)\in\Omega\times [0,T]$. 
\end{remark}
\begin{proof}
We prove this conclusion by induction. For $i=1$, we choose $\chi_2=v_{q, h}^{n-1}$ with $1 < q \leq k $ in the second equation of (\ref{app}) to get
\begin{align*}
& \gamma^{-1}\big(\frac{ v^{n}_{1, h}-v^{n-1}_{1, h}}{\tau}, v_{q, h}^{n-1}\big)=-(v_{1, h}^{n-\frac{1}{2}}, v_{1, h}^{n-\frac{1}{2}})\big[ (\mathbf{a} \nabla v^{n-\frac{1}{2}}_{1, h}, \nabla v_{q, h}^{n-1})-(f'(u_h^n)v_{1, h}^{n-\frac{1}{2}}, v_{q, h}^{n-1}) \big]\\[0.05in]
& \ds \qquad +(v_{1, h}^{n-\frac{1}{2}}, v_{q, h}^{n-1})\big[(\mathbf{a} \nabla v_{1, h}^{n-\frac{1}{2}}, \nabla v_{1, h}^{n-\frac{1}{2}})-(f'(u_h^n)v_{1, h}^{n-\frac{1}{2}}, v_{1, h}^{n-\frac{1}{2}})\big]\\[0.05in]
& \ds \qquad +\sum_{l> 1}(v_{1, h}^{n-\frac{1}{2}}, v_{1, h}^{n-\frac{1}{2}})(v_{l,h}^{n-1}, v_{q, h}^{n-1})\big[(\mathbf{a}\nabla v_{1, h}^{n-\frac{1}{2}}, \nabla v_{l, h}^{n-1})-(f'(u_h^n)v_{1, h}^{n-\frac{1}{2}}, v_{l, h}^{n-1})\big],
\end{align*}
which, together with $(v_{i, h}^{n-1}, v_{j, h}^{n-1})=\delta_{ij}$ for $1\leq i,j\leq k$,
leads to
\begin{align*}
(\tau \gamma)^{-1}\big(v^{n}_{1, h}, v_{q, h}^{n-1}\big)=\frac{1}{2}(v_{1, h}^{n}, v_{q, h}^{n-1})\big[(\mathbf{a} \nabla v_{1, h}^{n-\frac{1}{2}}, \nabla v_{1, h}^{n-\frac{1}{2}})-(f'(u_h^n)v_{1, h}^{n-\frac{1}{2}}, v_{1, h}^{n-\frac{1}{2}})\big].
\end{align*}
By (\ref{unless}), we have $(v_{1,h}^{n}, v_{q,h}^{n-1})=0$ for $1 < q \leq k$. 

Suppose that, for a given $1\leq p\leq k$, $(v_{i,h}^{n}, v_{q,h}^{n-1})=0$ for $ i<q\leq k$ and $1\leq i\leq p-1$. Then we intend to show the case of $i=p$, that is, $(v_{p,h}^{n}, v_{q,h}^{n-1})=0$ for $ p<q\leq k$. To prove this, we select $ \chi_2= v_{q ,h}^{n-1}$ for $p< q\leq k $ in the second equation of (\ref{app}) with $i=p$ to obtain
\begin{align*}
& \gamma^{-1}\big(\frac{ v^{n}_{p, h}-v^{n-1}_{p, h}}{\tau}, v_{q, h}^{n-1}\big)=-(v_{p, h}^{n-\frac{1}{2}}, v_{p, h}^{n-\frac{1}{2}})\big[ (\mathbf{a}\nabla v^{n-\frac{1}{2}}_{p, h}, \nabla v_{q, h}^{n-1})-(f'(u_h^n)v_{p, h}^{n-\frac{1}{2}}, v_{q, h}^{n-1}) \big]\\[0.05in]
& \ds \qquad +(v_{p, h}^{n-\frac{1}{2}}, v_{q, h}^{n-1})\big[(\mathbf{a}\nabla v_{p, h}^{n-\frac{1}{2}}, \nabla v_{p, h}^{n-\frac{1}{2}})-(f'(u_h^n)v_{p, h}^{n-\frac{1}{2}}, v_{p, h}^{n-\frac{1}{2}})\big]\\[0.05in]
& \ds \qquad +\sum_{l < p}(v_{p, h}^{n-\frac{1}{2}}, v_{p, h}^{n-\frac{1}{2}})(v_{l,h}^{n}, v_{q, h}^{n-1})\big[(\mathbf{a}\nabla v_{p, h}^{n-\frac{1}{2}}, \nabla v_{l, h}^{n})-(f'(u_h^n)v_{p, h}^{n-\frac{1}{2}}, v_{l, h}^{n})\big]\\[0.05in]
& \ds \qquad +\sum_{l> p}(v_{p, h}^{n-\frac{1}{2}}, v_{p, h}^{n-\frac{1}{2}})(v_{l,h}^{n-1}, v_{q, h}^{n-1})\big[(\mathbf{a} \nabla v_{p, h}^{n-\frac{1}{2}}, \nabla v_{l, h}^{n-1})-(f'(u_h^n)v_{p, h}^{n-\frac{1}{2}}, v_{l, h}^{n-1})\big],
\end{align*}
which, together with $(v_{i, h}^{n-1}, v_{j, h}^{n-1})=\delta_{ij}$ for $1\leq i,j\leq k$ and the induction hypothesis,
leads to
\begin{align*}
&(\tau \gamma)^{-1}\big(v^{n}_{p, h}, v_{q, h}^{n-1}\big)=\frac{1}{2}(v_{p, h}^{n}, v_{q, h}^{n-1})\big[(\mathbf{a} \nabla v_{p, h}^{n-\frac{1}{2}}, \nabla v_{p, h}^{n-\frac{1}{2}})-(f'(u_h^n)v_{p, h}^{n-\frac{1}{2}}, v_{p, h}^{n-\frac{1}{2}})\big].
\end{align*}
We invoke (\ref{unless}) to get $(v_{p,h}^{n}, v_{q,h}^{n-1})=0$ for $ p<q\leq k$, which completes the induction procedure and thus the whole proof.
\end{proof}

Now we prove the orthonormality of the scheme (\ref{app}).

\begin{theorem}\label{orthf}
	Suppose (\ref{unless}) holds for $1\leq n\leq N$. Then the numerical solutions of the scheme (\ref{app}) satisfy
	$(v_{i, h}^n,v_{j, h}^n)=\delta_{ij}$ for $1\leq i,j \leq k$ and $0\leq  n\leq N$.
\end{theorem}
\begin{proof}
 Since the initial values are orthonormal, the conclusion is valid for $n=0$. We then prove by induction.  Suppose $(v_{i, h}^{n-1}, v_{j, h}^{n-1})=\delta_{ij}$ for $1\leq i,j\leq k$, and we intend to prove 
$
 (v_{i, h}^{n}, v_{j, h}^{n})=\delta_{ij}
$ for $1\leq i,j\leq k$.
 
We take $\chi_2=v_{1, h}^{n-\frac{1}{2}}$ in the second equation of (\ref{app}) with $i=1$ and apply Lemma \ref{jump}, which implies 
 $(v_{l,h}^{n-1}, v_{1, h}^{n-\frac{1}{2}})=\frac{1}{2}(v_{l,h}^{n-1}, v_{1, h}^{n}+v_{1, h}^{n-1})=0$ for $1<l\leq k$,
to get $\gamma^{-1}(v_{1,h}^{n}-v_{1,h}^{n-1},v_{1,h}^{n-\frac{1}{2}})=0$, that is,
$\|v_{1,h}^{n}\|=1$. 
We then take $\chi_2=v_{1,h }^{n}$ in the second equation of (\ref{app}) with $i=2$ to obtain
\begin{equation*}
	\begin{aligned}
&\ds \gamma^{-1}\big(\frac{ v^{n}_{2, h}-v^{n-1}_{2, h}}{\tau}, v_{1,h }^{n}\big)=-(v_{2, h}^{n-\frac{1}{2}}, v_{2, h}^{n-\frac{1}{2}})\big[ (\mathbf{a} \nabla v^{n-\frac{1}{2}}_{2, h}, \nabla v_{1,h }^{n})-(f'(u_h^n)v_{2, h}^{n-\frac{1}{2}}, v_{1,h }^{n}) \big]\\
&\ds \qquad +(v_{2, h}^{n-\frac{1}{2}}, v_{1,h }^{n})\big[(\mathbf{a} \nabla v_{2, h}^{n-\frac{1}{2}}, \nabla v_{2, h}^{n-\frac{1}{2}})-(f'(u_h^n)v_{2, h}^{n-\frac{1}{2}}, v_{2, h}^{n-\frac{1}{2}})\big]\\
&\ds \qquad +\sum_{l< 2}(v_{2, h}^{n-\frac{1}{2}}, v_{2, h}^{n-\frac{1}{2}})(v_{l,h}^{n}, v_{1,h }^{n})\big[(\mathbf{a} \nabla v_{2, h}^{n-\frac{1}{2}}, \nabla v_{l, h}^{n})-(f'(u_h^n)v_{2, h}^{n-\frac{1}{2}}, v_{l, h}^{n})\big]\\
&\ds \qquad +\sum_{l>2}(v_{2, h}^{n-\frac{1}{2}}, v_{2, h}^{n-\frac{1}{2}})(v_{l,h}^{n-1}, v_{1,h }^{n})\big[(\mathbf{a} \nabla v_{2, h}^{n-\frac{1}{2}}, \nabla v_{l, h}^{n-1})-(f'(u_h^n)v_{2, h}^{n-\frac{1}{2}}, v_{l, h}^{n-1})\big],
\end{aligned}
\end{equation*} 
which, together with Lemma \ref{jump}, leads to
\begin{equation*}
	\begin{aligned}
& (\tau \gamma)^{-1}\big( v^{n}_{2, h}, v_{1,h}^{n}\big)=\frac{1}{2}(v_{2, h}^{n}, v_{1,h }^{n})\big[(\mathbf{a} \nabla v_{2, h}^{n-\frac{1}{2}}, \nabla v_{2, h}^{n-\frac{1}{2}})-(f'(u_h^n)v_{2, h}^{n-\frac{1}{2}}, v_{2, h}^{n-\frac{1}{2}})\big].
	\end{aligned}
\end{equation*} 
By (\ref{unless}), we have 
$
(v_{1,h}^{n}, v_{2,h}^{n})=0.
$
Furthermore, we select $\chi_2=v_{2,h}^{n-\frac{1}{2}}$ in (\ref{app}) with $i=2$  and employ Lemma \ref{jump} and $
(v_{1,h}^{n}, v_{2,h}^{n})=0
$ to get $\gamma^{-1}(v_{2,h}^{n}-v_{2,h}^{n-1},v_{2,h}^{n-1/2})=0$, which implies $\|v_{2,h}^{n}\|=1$. 

Now we summarize the obtained results as
$
	(v_{p, h}^n, v_{j, h}^n)=\delta_{pj}$ for $ 1\leq  j \leq p \leq 2.
$
Based on this, we could apply an inner induction, i.e., suppose that for some $i\leq k$,
\begin{equation*}
	(v_{p, h}^n, v_{j, h}^n)=\delta_{pj}\text{ for } 1\leq  j \leq p \leq i-1,
\end{equation*}
and we intend to prove the case of $p=i$, i.e., $(v_{i, h}^n, v_{j, h}^n)=\delta_{ij}$ for $ 1\leq  j \leq i$.
For $1\leq j<i$, we set $\chi_2= v_{j,h}^{n}$ in the second equation of (\ref{app}) to obtain 
\begin{equation*}
\begin{aligned}
& (\tau \gamma)^{-1}(v_{i,h}^{n}, v_{j,h}^{n}) =-(v_{i, h}^{n-\frac{1}{2}}, v_{i, h}^{n-\frac{1}{2}})\big[ (\mathbf{a} \nabla v^{n-\frac{1}{2}}_{i, h}, \nabla  v_{j,h}^{n} )-(f'(u_h^n)v_{i, h}^{n-\frac{1}{2}},  v_{j,h}^{n} ) \big]\\
& \ds \quad +(v_{i, h}^{n-\frac{1}{2}},  v_{j,h}^{n})\big[(\mathbf{a} \nabla v_{i, h}^{n-\frac{1}{2}}, \nabla v_{i, h}^{n-\frac{1}{2}})-(f'(u_h^n)v_{i, h}^{n-\frac{1}{2}}, v_{i, h}^{n-\frac{1}{2}})\big]\\
& \ds \quad +\sum_{l<i}(v_{i, h}^{n-\frac{1}{2}}, v_{i, h}^{n-\frac{1}{2}})(v_{l,h}^{n},  v_{j,h}^{n})\big[(\mathbf{a} \nabla v_{i, h}^{n-\frac{1}{2}}, \nabla v_{l, h}^{n})-(f'(u_h^n)v_{i, h}^{n-\frac{1}{2}}, v_{l, h}^{n})\big]\\
& \ds \quad +\sum_{l>i}(v_{i, h}^{n-\frac{1}{2}}, v_{i, h}^{n-\frac{1}{2}})(v_{l,h}^{n-1},  v_{j,h}^{n})\big[(\mathbf{a} \nabla v_{i, h}^{n-\frac{1}{2}}, \nabla v_{l, h}^{n-1})-(f'(u_h^n)v_{i, h}^{n-\frac{1}{2}}, v_{l, h}^{n-1})\big],
\end{aligned}
\end{equation*}
which, together with Lemma \ref{jump}, leads to
\begin{equation*}
\begin{aligned}
& (\tau \gamma)^{-1}(v_{i,h}^{n}, v_{j,h}^{n}) = \frac{1}{2}(v_{i, h}^{n},  v_{j,h}^{n})\big[(\mathbf{a}\nabla v_{i, h}^{n-\frac{1}{2}}, \nabla v_{i, h}^{n-\frac{1}{2}})-(f'(u_h^n)v_{i, h}^{n-\frac{1}{2}}, v_{i, h}^{n-\frac{1}{2}})\big].
	\end{aligned}
\end{equation*}
From (\ref{unless}), it follows that $ (v_{i,h}^{n}, v_{j,h}^{n}) =0$ for $1\leq j<i$. We then set $\chi_2= v_{i,h}^{n-\frac{1}{2}}$ in the second equation of (\ref{app}) and apply Lemma \ref{jump} to obtain  $(v_{i,h}^{n}-v_{i,h}^{n-1},v_{i,h}^{n-1/2})=0$, that is,
$\|v_{i,h}^{n}\|=1$. Thus, we complete the inner induction to obtain
$
	(v_{p, h}^n, v_{j, h}^n)=\delta_{p,j}$ for $ 1\leq  j \leq p\leq k.
$
 Consequently, the outer induction has also been completed and we have reached the conclusion.
\end{proof}

\section{Estimates of gradients}\label{sec4} 
Concerning the coupling and the gradient nonlinearity of the scheme (\ref{app}), it is difficult to prove the gradient stability for (\ref{app}). For this reason, in the study of the index-1 case \cite{ZhaZheZhu}, 
the numerical analysis is performed based on the assumption $\|u_h^n\|_{L^\infty}\leq Q$ for $0\leq n\leq N$. In the current work, we circumvent this 
assumption by employing the favorable properties of the scheme (\ref{app}), which reduces the difficulties of handling the coupling and gradient nonlinearity. Specifically, the numerical analysis can be performed based on an alternative condition  $u\in L^\infty(0,T;L^\infty)$, a natural assumption in conventional  numerical analysis.

Suppose  $u\in L^\infty(0,T;L^\infty)$. Since $f$ and $f'$ are local Lipschitz functions as assumed in the \textit{Assumption A}, we can define cutoff functions $\tilde f(z)$ and $\widetilde{f'}(z)$ of $f(z)$ and $f'(z)$, respectively, which equals to $f(z)$ and $f'(z)$, respectively, for $z\in [-\|u\|_{L^\infty(0,T;L^\infty)}-1,\|u\|_{L^\infty(0,T;L^\infty)}+1]$ and satisfy $|\tilde f|+|\widetilde{f'}|\leq Q$ and the global Lipschitz condition over $z\in\mathbb R$. Then we define an auxiliary scheme (denoted by \textit{AUX scheme} in the rest of the work) for the original scheme (\ref{app}), which coincides with (\ref{app}) with $f$ and $f'$ replaced by $\tilde f$ and $\widetilde{f'}$, respectively, and $\{u_h^n, v_{i,h}^n\}_{n=1}^{N}$ replaced by $\{\tilde u_h^n, \tilde v_{i,h}^n\}_{n=1}^{N}$ for $1\leq i\leq k$.

Following the same derivations in Section \ref{sec32}, the orthonormality preservation stated in Theorem \ref{orthf} still holds for the AUX scheme.
In subsequent contents, we still use the notations $u_{h}^n$ and $v_{i,h}^n$ to represent the solutions $\tilde u_{h}^n$ and $\tilde v_{i,h}^n$, respectively, of the AUX scheme for simplicity until Theorem \ref{thm}.
\begin{theorem}\label{stable}
Under the Assumption A and  (\ref{unless}) with $1\leq n\leq N$, we have $
	 \|\nabla u_h^n\|+\sum_{i=1}^k\|\nabla v_{i, h}^n\| \leq Q$ for the solutions of the AUX scheme for $\tau$  small enough.
\end{theorem}
\begin{proof} 
We first take $\chi_2= v_{i,h}^{n}-v_{i,h}^{n-1}$ in the second equation of the AUX scheme and apply Lemma \ref{jump}  and Theorem \ref{orthf}, which implies $(v_{l,h}^{n}, v_{i,h}^{n}-v_{i,h}^{n-1})=0$ for $l<i$ and $(v_{l,h}^{n-1}, v_{i,h}^{n}-v_{i,h}^{n-1})=0$ for $l> i$, to get
\begin{equation*}
	\begin{aligned}
&\ds \gamma^{-1}\big(\frac{ v^{n}_{i, h}-v^{n-1}_{i, h}}{\tau}, v_{i,h}^{n}-v_{i,h}^{n-1} \big)\\
&=-(v_{i, h}^{n-\frac{1}{2}}, v_{i, h}^{n-\frac{1}{2}})\big[ (\mathbf{a} \nabla v^{n-\frac{1}{2}}_{i, h}, \nabla v_{i,h}^{n}-\nabla v_{i,h}^{n-1})-(\widetilde{f'}(u_h^n)v_{i, h}^{n-\frac{1}{2}}, v_{i,h}^{n}-v_{i,h}^{n-1}) \big]\\[0.1in]
&\ds \quad +(v_{i, h}^{n-\frac{1}{2}}, v_{i,h}^{n}-v_{i,h}^{n-1})\big[(\mathbf{a} \nabla v_{i, h}^{n-\frac{1}{2}}, \nabla v_{i, h}^{n-\frac{1}{2}})-(\widetilde{f'}(u_h^n)v_{i, h}^{n-\frac{1}{2}}, v_{i, h}^{n-\frac{1}{2}})\big]\\[0.1in]
&\ds \quad +\sum_{l<i}(v_{i, h}^{n-\frac{1}{2}}, v_{i, h}^{n-\frac{1}{2}})(v_{l,h}^{n}, v_{i,h}^{n}-v_{i,h}^{n-1})\big[(\mathbf{a} \nabla v_{i, h}^{n-\frac{1}{2}}, \nabla v_{l, h}^{n})-(\widetilde{f'}(u_h^n)v_{i, h}^{n-\frac{1}{2}}, v_{l, h}^{n})\big]\\
&\ds \quad +\sum_{l > i}(v_{i, h}^{n-\frac{1}{2}}, v_{i, h}^{n-\frac{1}{2}})(v_{l,h}^{n-1}, v_{i,h}^{n}-v_{i,h}^{n-1})\big[(\mathbf{a} \nabla v_{i, h}^{n-\frac{1}{2}}, \nabla v_{l, h}^{n-1})-(\widetilde{f'}(u_h^n)v_{i, h}^{n-\frac{1}{2}}, v_{l, h}^{n-1})\big]\\
&=-(v_{i, h}^{n-\frac{1}{2}}, v_{i, h}^{n-1})\big[ (\mathbf{a} \nabla v^{n-\frac{1}{2}}_{i, h}, \nabla v_{i,h}^{n})-(\widetilde{f'}(u_h^n)v_{i, h}^{n-\frac{1}{2}}, v_{i,h}^{n}) \big]\\[0.1in]
&\qquad +(v_{i, h}^{n-\frac{1}{2}}, v_{i,h}^{n})\big[(\mathbf{a} \nabla v_{i, h}^{n-\frac{1}{2}}, \nabla v_{i, h}^{n-1})-(\widetilde{f'}(u_h^n)v_{i, h}^{n-\frac{1}{2}}, v_{i, h}^{n-1})\big]\\[0.1in]
&=(v_{i, h}^{n-\frac{1}{2}}, v_{i,h}^{n}) \Big( \frac{1}{2} \|\sqrt{\mathbf{a}}\nabla v_{i, h}^{n-1}\|^2-\frac{1}{2}\|\sqrt{\mathbf{a}}\nabla v_{i, h}^{n}\|^2 + (\widetilde{f'}(u_h^n)v_{i, h}^{n-\frac{1}{2}}, v_{i, h}^{n}-v_{i, h}^{n-1})\Big),
	\end{aligned}
\end{equation*}
where we have used
$(v_{i, h}^{n-\frac{1}{2}}, v_{i, h}^{n-1})=[1+(v_{i, h}^{n-1},v_{i, h}^{n})]/2=(v_{i, h}^{n-\frac{1}{2}}, v_{i,h}^{n}) $ in the last equality.
We then invoke the boundedness of $\widetilde{f'}$ to derive 
\begin{equation}\label{s1}
 \gamma^{-1} \|v^{n}_{i, h}-v^{n-1}_{i, h}\|^2 \leq  -\frac{\tau}{2}(v_{i, h}^{n-\frac{1}{2}}, v_{i,h}^{n})(\|\sqrt{\mathbf{a}} \nabla v_{i, h}^{n}\|^2-\|\sqrt{\mathbf{a}} \nabla v_{i, h}^{n-1}\|^2)+Q\tau \| v_{i, h}^{n}-v_{i, h}^{n-1}\|.
\end{equation}  
 If $\|\sqrt{\mathbf{a}} \nabla v_{i, h}^{n}\|^2-\|\sqrt{\mathbf{a}} \nabla v_{i, h}^{n-1}\|^2\leq 0$, then  $\|\sqrt{\mathbf{a}} \nabla v_{i, h}^{n}\|^2-\|\sqrt{\mathbf{a}} \nabla v_{i, h}^{n-1}\|^2\leq Q\tau$ for any $Q>0$. Otherwise,  if $\|\sqrt{\mathbf{a}} \nabla v_{i, h}^{n}\|^2-\|\sqrt{\mathbf{a}} \nabla v_{i, h}^{n-1}\|^2> 0$,
we apply
\begin{equation}\label{jy15}
(v_{i, h}^{n-\frac{1}{2}}, v_{i,h}^{n})=\frac{1+(v_{i,h}^{n-1},v_{i,h}^{n})}{2}=\frac{\|v^{n}_{i, h}+v^{n-1}_{i, h}\|^2}{4}\geq 0
\end{equation}
to get from (\ref{s1}) that $\gamma^{-1} \|v^{n}_{i, h}-v^{n-1}_{i, h}\|^2 \leq  Q\tau \| v_{i, h}^{n}-v_{i, h}^{n-1}\|$, that is, 
$
 \|v^{n}_{i, h}-v^{n-1}_{i, h}\| \leq  Q\tau.
$
We combine this with
\begin{align}\label{Zjy1}
(v_{i, h}^{n-\frac{1}{2}}, v_{i, h}^{n})=(v_{i, h}^{n-\frac{1}{2}}, v_{i, h}^{n-\frac{1}{2}})=1- \frac{1-(v_{i, h}^{n},v_{i, h}^{n-1})}{2}=1-\frac{\|v^{n}_{i, h}-v^{n-1}_{i, h}\|^2}{4}
 \end{align}
to get $( v_{i, h}^{n-\frac{1}{2}}, v_{i, h}^{n})\geq 1/2$ for $\tau$ small enough. We invoke the above estimates in (\ref{s1}) to obtain
\begin{equation*}
\frac{\tau}{4}(\|\sqrt{\mathbf{a}} \nabla v_{i, h}^{n}\|^2-\|\sqrt{\mathbf{a}} \nabla v_{i, h}^{n-1}\|^2)\leq Q\tau^2, \text{ i.e., }\|\sqrt{\mathbf{a}} \nabla v_{i, h}^{n}\|^2-\|\sqrt{\mathbf{a}} \nabla v_{i, h}^{n-1}\|^2\leq Q\tau.
\end{equation*}  
Consequently,
in either case we have $\|\sqrt{\mathbf{a}}\nabla v_{i, h}^{n}\|^2-\|\sqrt{\mathbf{a}}\nabla v_{i, h}^{n-1}\|^2\leq Q\tau$. We sum this equation from $n = 1$ to $n^* \leq N$ and invoke \textit {Assumption A} to get
$$
a_0\|\nabla v_{i,h}^{n^*}\|^2 \leq \|\sqrt{\mathbf{a}} \nabla v_{i,h}^{n^*}\|^2 \leq \|\sqrt{\mathbf{a}}\nabla v_{i,h}^{0}\|^2+ Q\tau n^* \leq Q\|\nabla v_{i,0}\|^2+ QT\leq Q.
$$

To bound $\nabla u_h^{n}$, we take $\chi_1 =u_h^n-u_h^{n-1}$ in the first equation of the AUX scheme and use the boundedness of $\|u_h^n\|$ and $\|\nabla v_{i,h}^n\|$ to obtain
\begin{equation}\label{Bu}
\begin{aligned}
&\beta^{-1}\|u_h^n-u_h^{n-1}\|^2 + \tau\|\sqrt{\mathbf{a}} \nabla u_h^n\|^2 \leq \tau\|\sqrt{\mathbf{a}} \nabla u_h^n\|\cdot \|\sqrt{\mathbf{a}} \nabla u_h^{n-1}\| + Q\tau\|u_h^n-u_h^{n-1}\|\\
&\qquad +2\tau \sum_{i=1}^k\|\sqrt{\mathbf{a}}\nabla v_{i,h}^{n-1}\|\cdot\|\sqrt{\mathbf{a}}\nabla u_h^{n}\|\cdot\|u_h^n-u_h^{n-1}\|\\
&\quad \leq  \frac{\tau}{2}\|\sqrt{\mathbf{a}} \nabla u_h^n\|^2  +\frac{\tau}{2}\|\sqrt{\mathbf{a}}\nabla u_h^{n-1}\|^2+Q\tau^2\|\sqrt{\mathbf{a}} \nabla u_h^{n}\|^2+\beta ^{-1}\|u_h^n-u_h^{n-1}\|^2+Q\tau^2,
\end{aligned}
\end{equation}
which implies
$
	\|\sqrt{\mathbf{a}}\nabla u_h^n\|^2-\|\sqrt{\mathbf{a}}\nabla u_h^{n-1}\|^2\leq Q\tau\|\sqrt{\mathbf{a}}\nabla u_h^{n}\|^2+Q\tau.
$
We sum this equation with respect to $n$, then apply the Gronwall inequality and \textit{Assumption A} to get $a_0 \|\nabla u_h^n \|^2 \leq \|\sqrt{\mathbf{a}} \nabla u_h^n \|^2 \leq Q $, which completes the proof.
\end{proof}

\section{ Error estimates and index preservation}\label{er}
Define $e_u^{n}:=u(t_n)-u_h^n$ and $  e_{v_i}^{n}:=v_i(t_n)-v_{i,h}^n$.  Combining Theorem \ref{stable} and (\ref{s1}) we have $\|v^{n}_{i, h}-v^{n-1}_{i, h}\|\leq Q\sqrt{\tau}$. We invoke this in (\ref{Zjy1})
to derive
 \begin{align}\label{zjy20}
  ( v_{i, h}^{n-\frac{1}{2}}, v_{i, h}^{n-\frac{1}{2}})=1+O(\tau).
  \end{align}
Applying this result, we have
\begin{equation*}
	\begin{aligned}
&(v_i(t_n), v_i(t_n))(\mathbf{a} \nabla v_i(t_{n}), \nabla \chi_2)-(v_{i, h}^{n-\frac{1}{2}}, v_{i, h}^{n-\frac{1}{2}})(\mathbf{a} \nabla v^{n-\frac{1}{2}}_{i, h}, \nabla \chi_2)\\[0.05in]
& =(\mathbf{a} \nabla v_i(t_{n}), \nabla \chi_2)- (1+O(\tau))(\mathbf{a} \nabla v^{n-\frac{1}{2}}_{i, h}, \nabla \chi_2)\\[0.05in]
&=\frac{(\mathbf{a}(\nabla e_{v_i}^n+\nabla e_{v_i}^{n-1}), \nabla \chi_2)+ (\mathbf{a}(\nabla v_i(t_{n})-\nabla v_i(t_{n-1})), \nabla \chi_2) }{2}\!+\! O(\tau)(\mathbf{a}\nabla v^{n-\frac{1}{2}}_{i, h}, \nabla \chi_2).
	\end{aligned}
\end{equation*}
Subtracting the AUX scheme from (\ref{ref}) and using the above relation, we obtain 
\begin{equation}\label{error}
\left\{
\begin{array}{l}
\displaystyle  \beta^{-1}\big(\frac{e^n_u-e^{n-1}_u}{\tau}+R_u^n, \chi_1\big)  +  (\mathbf{a} \nabla e^{n}_u, \nabla \chi_1)  +  \mu_1^n  =   \mu_2^n  + \mu_3^n ,\\[0.1in]
\displaystyle \gamma^{-1}\big(\frac{ e^{n}_{v_i}-e^{n-1}_{v_i}}{\tau}+R_{v_i}^n, \chi_2\big)+\frac{1}{2}(\mathbf{a}(\nabla e_{v_i}^n+\nabla e_{v_i}^{n-1}), \nabla \chi_2)=:\sum_{m=1}^4\nu_{i,m}^n ,\end{array}
\right.
\end{equation}
for $1\leq i\leq k$, where
$ R_u^n=u_t(t_n)-\bar\partial_t u(t_n)$, $R_{v_i}^n=(v_{i})_t(t_n)-\bar\partial_t v_i(t_n)$
 and the truncation errors $\mu_1^n$--$\mu_3^n$ and $\nu_1^n$--$\nu_4^n$ are presented  in the Appendix due to lengthy expressions.

Then we derive estimates based on the \textit{Assumption A} and the condition (\ref{unless}) with $1\leq i\leq N$ (such that Theorems \ref{orthf}--\ref{stable} can be applied) and the regularity condition $u,v\in H^1(0,T;H^2)\cap H^2(0,T;L^2)$. Note that such regularity conditions imply that $u$ and $v$ belong to $ C([0,T];L^\infty)$, $C([0,T];H^2)$ and $ W^{1,\infty}(0,T;L^2)$ by Sobolev embedding, and $(\nabla v_i)_t\in L^2(0,T;H^1_0)\cap H^1(0,T;H^{-1})$, which, according to \cite[Section 5.9.2, Theorem 3]{Eva}, gives $(\nabla v_i)_t\in C([0,T];L^2)$. 
Furthermore, we decompose 
\begin{equation*}
	\begin{aligned}
u(t_n)- u_h^n &= u(t_n)-P u(t_n)+ P u(t_n)-u_h^n=: \eta_u^n+\xi_u^n,\\
v_i(t_n)- v_{i, h}^n & = v_i(t_n)-P v_i(t_n) +P v_i(t_n) -v_{i, h}^n=:\eta_{v_i}^n+\xi_{v_i}^n.
	\end{aligned}
\end{equation*}

 We first invoke these splittings and the estimates of $\mu_1^n$--$\mu_3^n$ (cf. Appendix) in the error equation of $u$ with $\chi_1=\xi_u^n$ to get
\begin{equation}\label{u1}
\begin{aligned}
& \beta^{-1}\|\xi_u^n\|^2+\tau\|\sqrt{\mathbf{a}} \nabla\xi_u^n\|^2 \leq \beta^{-1} \|\xi_u^{n-1}\|\cdot\|\xi_u^{n}\|+\tau\|\psi^n_u\|\cdot\|\xi_u^{n}\|+Q\tau^2 \|\xi_u^n\| \\
& \quad+Q\tau \|\sqrt{\mathbf{a}} \nabla \xi_u^{n}\| \cdot\|\xi_u^{n}\| +Q\tau\|e_u^{n-1}\|\cdot\|\xi_u^{n}\| + Q\tau\sum_{i=1}^k\| e_{v_i}^{n-1}\|\cdot\|\xi_u^{n}\|,
\end{aligned}
\end{equation}
where $\psi^n_u:=\beta^{-1}(\frac{\eta_u^n-\eta_u^{n-1}}{\tau}+R^n_u)$. Following the standard procedure (see e.g. (4.13)--(4.14) in \cite{ZhaZheZhu}) we obtain 
\begin{equation}\label{xiu}
\|\xi_u^{n}\| \leq  Q\tau\sum_{m=1}^{n}\sum_{i=1}^k \|\xi_{v_i}^{m-1}\|+Q\tau+Qh^2,~~1\leq n\leq N.
\end{equation}

We then invoke the estimates of truncation errors $\nu_{i,1}^n$--$\nu_{i,4}^n$ (cf. Appendix for details) in the error equation of $v_i$ with $\chi_2= \xi_{v_i}^{n-\frac{1}{2}}$ to get 
\begin{equation}\label{v2}
	\begin{aligned}
& \gamma^{-1}\|\xi_{v_i}^{n}\|^2+\tau\|\sqrt{\mathbf{a}} \nabla \xi_{v_i}^{n-\frac{1}{2}}\|^2 \leq \gamma^{-1}\|\xi_{v_i}^{n-1}\|^2+Q\tau\| \psi_{v_i}^n\|\cdot \|\xi_{v_i}^{n-\frac{1}{2}}\|\\
&\qquad + Q\tau\sum_{j=1}^k(\|e_{v_j}^{n-1}\|+\|e_{v_j}^{n}\|)\cdot \|\xi_{v_i}^{n-\frac{1}{2}}\|+Q\tau \|\sqrt{\mathbf{a}}\nabla \xi_{v_i}^{n-\frac{1}{2}}\|\cdot \|\xi_{v_i}^{n-\frac{1}{2}}\|\\
&\qquad +Q\tau^2\|\sqrt{\mathbf{a}}\nabla \xi_{v_i}^{n-\frac{1}{2}}\|+Q\tau^2\|\xi_{v_i}^{n-\frac{1}{2}}\|+Q\tau \|e_u^n\|\cdot \|\xi_{v_i}^{n-\frac{1}{2}}\|.\\[0.05in]
	\end{aligned}
\end{equation}
where $\psi^n_{v_i}=-\gamma^{-1}(\frac{\eta_v^n-\eta_v^{n-1}}{\tau}+R^n_v)$. By Young’s inequality, we have
\begin{equation*}
	\begin{aligned}
Q\tau \|\sqrt{\mathbf{a}}\nabla \xi_{v_i}^{n-\frac{1}{2}}\|\cdot \|\xi_{v_i}^{n-\frac{1}{2}}\|+Q\tau^2\|\sqrt{\mathbf{a}} \nabla \xi_{v_i}^{n-\frac{1}{2}}\|  \leq  \frac{\tau}{2} \|\sqrt{\mathbf{a}}\nabla \xi_{v_i}^{n-\frac{1}{2}}\|^2+ Q\tau \|\xi_{v_i}^{n-\frac{1}{2}}\|^2+Q\tau^3,
	\end{aligned}
\end{equation*}
and we combine this inequality with  (\ref{v2}) and sum over $i$ from $1$ to $k$ to obtain
\begin{align*}
&\gamma^{-1} \sum_{i=1}^k \|\xi_{v_i}^{n}\|^2 \leq \gamma^{-1} \sum_{i=1}^k \|\xi_{v_i}^{n-1}\|^2+Q\tau \sum_{i=1}^k \| \psi_{v_i}^n\|\cdot \|\xi_{v_i}^{n-\frac{1}{2}}\|+Q\tau \sum_{i=1}^k \|e_u^n\|\cdot \|\xi_{v_i}^{n-\frac{1}{2}}\|\\
&\quad + Q\tau \sum_{i=1}^k \sum_{j=1}^k(\|e_{v_j}^{n-1}\|+\|e_{v_j}^{n}\|)\cdot \|\xi_{v_i}^{n-\frac{1}{2}}\|+ Q\tau \sum_{i=1}^k \|\xi_{v_i}^{n-\frac{1}{2}}\|^2+ Q\tau^3\\
&\leq \gamma^{-1} \sum_{i=1}^k \|\xi_{v_i}^{n-1}\|^2 +Q\tau \sum_{i=1}^k ( \| \psi_{v_i}^n\|^2+ \|e_u^n\|^2+ \|e_{v_i}^{n-1}\|^2 +\|e_{v_i}^{n}\|^2+\|\xi_{v_i}^{n-\frac{1}{2}}\|^2)+Q\tau^3.
\end{align*}
We sum this equation from $n=1$ to $n^*\leq N$ and invoke the standard estimates $\tau \sum_{n=1}^{N}\|\psi^n_{v_i}\|^2\leq Q\|(v_i)_{tt}\|_{L^2(0,T;L^2)}^2\tau^2+Q\|(v_i)_t\|^2_{L^2(0,T;H^2)}h^4$,
the projection estimate (\ref{eta1}) and the estimate  (\ref{jy8}) on initial values to get
\begin{equation}\label{xiv}
\begin{aligned}
\sum_{i=1}^k \|\xi_{v_i}^{n^*}\|^2  \leq Q\tau\sum_{n=1}^{n^*} \|\xi_{u}^{n}\|^2 + Q\tau \sum_{n=1}^{n^*} \sum_{i=1}^k \|\xi_{v_i}^{n}\|^2+Q\tau^2+Q h^4.
\end{aligned}
\end{equation}
 Substituting (\ref{xiu}) into (\ref{xiv}) we derive that
\begin{equation*}
\begin{aligned}
 \sum_{i=1}^k \|\xi_{v_i}^{n^*}\|^2  & \leq  Q\tau^2\sum_{n=1}^{n^*} \sum_{m=1}^{n}\sum_{i=1}^k \|\xi_{v_i}^{m-1}\|^2+ Q\tau \sum_{n=1}^{n^*} \sum_{i=1}^k \|\xi_{v_i}^{n}\|^2 +Q\tau^2+Q h^4\\
 & \leq  Q\tau \sum_{n=1}^{n^*} \sum_{i=1}^k \|\xi_{v_i}^{n}\|^2+Q\tau^2+Q h^4.
\end{aligned}
\end{equation*}
Thus for $\tau$ small enough, an application of the Gronwall inequality leads to
\begin{equation*}
	\sum_{i=1}^k \|\xi_{v_i}^{n}\|^2 \leq Q\tau^2+Qh^4,\text{ that is, } \sum_{i=1}^k \|\xi_{v_i}^{n}\| \leq Q\sqrt{k}(\tau+h^2),~1\leq n\leq N,
\end{equation*}
and we invoke this in (\ref{xiu}) to get $\|\xi_u^n\|\leq Q(\tau+h^2)$. Combining these results with the projection estimate (\ref{eta1}) leads to the following theorem.
\begin{theorem}\label{thma}
Under Assumption A, the condition (\ref{unless}) for $1\leq n\leq N$ and the regularity condition $u,v_i\in H^1(0,T;H^2)\cap H^2(0,T;L^2)$ for $1\leq i\leq k$, the following error estimate  holds for the AUX scheme with  $\tau$ small enough 
$$\| u(t_n)- u_h^n\|+\sum_{i=1}^k\|v_i(t_n)- v_{i,h}^n\|\leq Q(\tau+h^2),~~1\leq n\leq N.$$
\end{theorem} 

Based on the estimates for the AUX scheme, we can now give error estimates for the original scheme (\ref{app}). Since the numerical solutions of both the AUX scheme and the original scheme (\ref{app}) will appear in subsequent proofs, we reuse the notations $\tilde u_h^n$ and $\tilde v_{i,h}^n$ for the numerical solutions of the AUX scheme to avoid possible confusion.

\begin{theorem}\label{thm}
Under Assumption A, the condition (\ref{unless}) for $1\leq n\leq N$, the regularity condition $u,v_i\in H^1(0,T;H^2)\cap H^2(0,T;L^2)$ for $1\leq i\leq k$ and the time-step condition $\tau=o(h^{d/2})$, the following error estimate  holds for the original scheme (\ref{app}) with  $\tau$ and $h$ small enough 
$$\| u(t_n)- u_h^n\|+\sum_{i=1}^k\|v_i(t_n)- v_{i,h}^n\|\leq Q(\tau+h^2),~~1\leq n\leq N.$$
\end{theorem} 
\begin{remark}
If we only consider the spatial semi-discrete scheme, we do not need the  time-step condition $\tau=o(h^{d/2})$ in error estimates.
\end{remark}
\begin{proof}
Based on the interpolation estimates $\|q-I_h q\|_{L^\infty}\leq Qh^{2-d/2}\|q\|_{H^2}$ and $\|q-I_h q\|\leq Qh^{2}\|q\|_{H^2}$ for $q\in H^2$ and the inverse estimate $\|q_h\|_{L^\infty}\leq Qh^{-d/2}\|q_h\|$ for $q_h\in S_h$ \cite{Bre5}, we have
\begin{equation*}
	\begin{aligned}
	\displaystyle   \|u-\tilde u_h^n\|_{L^\infty} & \leq \| u-I_h u\|_{L^\infty}+\|I_h u- \tilde u_h^n\|_{L^\infty}\leq Qh^{2-d/2}\|u\|_{H^2}+Qh^{-d/2}\|I_h u- \tilde u_h^n\|\\[0.05in]
	\displaystyle  &\leq Qh^{2-d/2}+ Qh^{-d/2}\big(\|I_h u- u\|+\| u- \tilde u_h^n\|\big)\\
	& \leq Qh^{2-d/2}+Qh^{-d/2}(h^2+\tau)\leq   Qh^{2-d/2}+Qh^{-d/2}o(h^{d/2}).
	\end{aligned}
\end{equation*}
As the right-hand side of the above equation tends to $0$ as $h\rightarrow 0^+$, we find that $\|\tilde u_h^n\|_{L^\infty}\leq \|u\|_{L^\infty(0,T;L^\infty)}+1$ for $h$ small enough. In this case, the $\tilde f$ and $\widetilde{f'}$ are consistent with $f$ and $f'$, respectively, and the AUX scheme degenerates to the original scheme (\ref{app}), which, together with  Theorem \ref{thma}, completes the proof. 
\end{proof}

Based on the error estimate, we could give a theoretical support for the index-preservation issue of the scheme.
\begin{corollary}\label{cor1}
Under the conditions of Theorem \ref{thm}, if $\lim\limits_{t\rightarrow \infty}u(t)=u^*$ under the $L^2$ sense for an index-k saddle point $u^*$, then for any $\delta>0$, the approximation $\|u_h^N-u^*\|\leq \delta$ holds for $T$ large enough and $\tau$ and $h$ small enough.

In addition, if for any $\hat u\in B_\delta(u^*):=\{\hat u\in L^2:\,\|\hat u-u^*\|\leq \delta\}$ for some $\delta>0$ the eigenvalues of $-\nabla\cdot(\mathbf{a}(x)\nabla)-f'(\hat u)$ and $-\nabla\cdot(\mathbf{a}(x)\nabla)-f'( u^*)$ under zero Dirichlet boundary conditions have the same signs, then the eigenvalues of  $-\nabla\cdot(\mathbf{a}(x)\nabla)-f'(u_h^N)$ and $-\nabla\cdot(\mathbf{a}(x)\nabla)-f'( u^*)$ also have the same signs for $T$ large enough and $\tau$, $h$ small enough, that is, the scheme (\ref{app}) preserves the indices of saddle  solutions of (\ref{elliptic}). 
\end{corollary}
\begin{proof}
As $\lim\limits_{t\rightarrow \infty}u(t)=u^*$ under the $L^2$ sense, there exists a $T_0>0$ such that $\|u(T_0)-u^*\|\leq \delta/2$.  According to Theorem \ref{thm} with $T=T_0$,  we have $\|u_h^N- u(T_0)\|\leq Q (\tau + h^2)\leq \delta/2$ for $\tau$ and $h$ small enough. Thus we have $\|u_h^N-u^*\|\leq \delta$. The second statement of the theorem is a direct consequence of this estimate.
\end{proof}
\begin{remark}
The condition ``for any $\hat u\in B_\delta(u^*)$ for some $\delta>0$ the eigenvalues of $-\nabla\cdot(\mathbf{a}(x)\nabla)-f'(\hat u)$ and $-\nabla\cdot(\mathbf{a}(x)\nabla)-f'( u^*)$ under zero Dirichlet boundary conditions have the same signs'' in Corollary \ref{cor1} has been rigorously justified for a wide class of one-dimensional problems in \cite[Remark 3.3]{ZhaZheZhu}, which suggests the reasonableness of such assumption. 
Furthermore, Corollary \ref{cor1} implies that $u_h^N$ could approximate $u^*$ with an arbitrarily small error for $T$ sufficiently large and $\tau$ and $h$ small enough, which justifies the sufficiency of considering (\ref{SD}) on a finite interval.
\end{remark}
\section{Extension to advection-reaction-diffusion model}
We extend the preceding studies to the semilinear advection-reaction-diffusion equation 
\begin{equation}\label{RDC}
\mathcal Lu+f(u):=	\nabla\cdot(\mathbf{a}(x)\nabla u) + \mathbf{b}(x)\cdot\nabla u + c(x)u+f(u)=0,~x\in \Omega;~u=0,~x\in \partial \Omega,
\end{equation}
for some $ \mathbf{b}(x)\in\mathbb R^d$ and $ c(x)$ such that $\|\mathbf{b}\|_{L^\infty}+\|c\|_{L^\infty}\leq Q$. Then the corresponding index-$k$ saddle dynamics for locating index-$k$ saddle points reads
\begin{equation}\label{RSD}
\left\{\begin{array}{l}
\ds  u_t=\beta\Big(\mathcal Lu+f(u) - 2\sum_{i=1}^k v_i\big (v_i,\mathcal Lu+f(u) \big)\Big),\\
\ds (v_i)_t =\gamma\Big(\mathcal Lv_i+f'(u)v_i -\sum_{j=1}^{k}v_j(v_j, \mathcal Lv_i+f'(u)v_i )\Big),~~1\leq i\leq k,
    \end{array}
    \right.
\end{equation}
for $x\in\Omega$ and $t>0$, equipped with initial values $u_0$ and $\{v_{i,0}\}_{i=1}^k$ and zero boundary conditions. Similar to the proof of Theorem \ref{orth}, the following orthonormality holds.
 \begin{theorem}\label{orth2}
If the initial values $\{v_{i,0}\}_{i=1}^k$ of (\ref{RSD}) satisfy $(v_{i,0},v_{j,0})=\delta_{ij}$ for $ 1\leq i,j \leq k$, then $(v_{i},v_{j})=\delta_{ij}$ for $1\leq i,j \leq k$ for any $t>0$.
\end{theorem}

To obtain the retraction-free orthonormality-preserving scheme of (\ref{RSD}), we  denote $\mathcal F(u_h^{n-1}):= \mathbf{b}(x)\cdot\nabla u_h^{n-1} + c(x)u_h^{n-1}+ f(u_h^{n-1})$ and $\mathcal G^n:= \mathbf{b}(x)\cdot\nabla + c(x)+ f'(u_h^n)$. Then the scheme is given in an analogous manner as (\ref{app}): find $\{u_h^n, v_{i,h}^n\}_{n=1,i=1}^{N,k}$ such that for any $\chi_1,\chi_2 \in S_h$
\begin{equation}\left\{
\begin{array}{l}
 \ds \beta ^{-1}\big(\frac{u^n_h-u^{n-1}_h}{\tau}, \chi_1\big)+(\mathbf{a} \nabla u^{n}_h, \nabla \chi_1)-2\sum_{i=1}^k(v^{n-1}_{i, h}, \chi_1)(\mathbf{a}\nabla v^{n-1}_{i, h}, \nabla u^{n}_h)\\[0.0in]
\ds\quad  =\big(\mathcal F(u_h^{n-1}), \chi_1\big) -2\sum_{i=1}^{k}(v^{n-1}_{i, h},\chi_1)\big(\mathcal F(u_h^{n-1}), v^{n-1}_{i, h}\big),\\[0.2in]
\ds \gamma^{-1}\big(\frac{ v^{n}_{i, h}-v^{n-1}_{i, h}}{\tau}, \chi_2\big)=-(v_{i, h}^{n-\frac{1}{2}}, v_{i, h}^{n-\frac{1}{2}})\big [(\mathbf{a}\nabla v^{n-\frac{1}{2}}_{i, h}, \nabla \chi_2)- (\mathcal G^nv_{i, h}^{n-\frac{1}{2}}, \chi_2)\big]\\[0.15in]
\ds \quad +(v_{i, h}^{n-\frac{1}{2}}, \chi_2)\big[(\mathbf{a}\nabla v_{i, h}^{n-\frac{1}{2}}, \nabla v_{i, h}^{n-\frac{1}{2}})-( \mathcal G^n v_{i, h}^{n-\frac{1}{2}}, v_{i, h}^{n-\frac{1}{2}})\big]\label{Rapp}
\\[0.15in]
\ds \quad +\sum_{l< i}(v_{i, h}^{n-\frac{1}{2}}, v_{i, h}^{n-\frac{1}{2}})(v_{l,h}^{n}, \chi_2)\big[(\mathbf{a}\nabla v_{i, h}^{n-\frac{1}{2}}, \nabla v_{l, h}^{n})-(\mathcal G^n v_{i, h}^{n-\frac{1}{2}}, v_{l, h}^{n})\big]\\[0.15in]
\ds \quad +\sum_{l> i}(v_{i, h}^{n-\frac{1}{2}}, v_{i, h}^{n-\frac{1}{2}})(v_{l,h}^{n-1}, \chi_2)\big[(\mathbf{a}\nabla v_{i, h}^{n-\frac{1}{2}}, \nabla v_{l, h}^{n-1})-(\mathcal G^n v_{i, h}^{n-\frac{1}{2}}, v_{l, h}^{n-1})\big],
\end{array}
\right.
\end{equation}
where the initial values for scheme (\ref{Rapp}) are the same as those for (\ref{app}). 

The orthonormality of the scheme (\ref{Rapp}) could be proved in the same manner as (\ref{app}). Thus we only state the result.
\begin{theorem}\label{orthR}
	Suppose  for $1\leq n\leq N$
	\begin{equation}\label{unless2}
(\tau \gamma)^{-1} \neq \frac{1}{2}\big[(\mathbf{a}\nabla v_{i, h}^{n-\frac{1}{2}}, \nabla v_{i, h}^{n-\frac{1}{2}})-( \mathcal G^n v_{i, h}^{n-\frac{1}{2}}, v_{i, h}^{n-\frac{1}{2}})\big],~~1\leq i\leq k,
\end{equation}
	then the numerical solutions of the scheme (\ref{Rapp}) satisfy $(v_{i, h}^n,v_{j, h}^n)=\delta_{ij}$ for $1\leq i,j \leq k$ and $ 1\leq n\leq N$.
\end{theorem}

For the sake of further analysis, we again introduce
 an auxiliary scheme (denoted by \textit{AUX* scheme} in the rest of the work) for the original scheme (\ref{Rapp}), which coincides with (\ref{Rapp}) with $\mathcal F(u_h^{n-1})$ and $\mathcal G^n$ replaced by $\widetilde{\mathcal F}(u_h^{n-1}):= \mathbf{b}(x)\cdot\nabla u_h^{n-1} + c(x)u_h^{n-1}+ \tilde f(u_h^{n-1})$ and $\widetilde{\mathcal G}^n:= \mathbf{b}(x)\cdot\nabla + c(x)+ \widetilde{f'}(u_h^n)$, respectively, and $\{u_h^n, v_{i,h}^n\}_{n=1}^{N}$ replaced by $\{\tilde u_h^n, \tilde v_{i,h}^n\}_{n=1}^{N}$ for $1\leq i\leq k$. Here $\tilde f$ and $\widetilde{f'}$ denote the cutoff functions  presented in Section \ref{sec4}.

Now we are in the position to prove gradient estimates of the numerical solutions for the AUX* scheme.

\begin{theorem}\label{stable2}
Under the Assumption A and the condition (\ref{unless2}) with $1\leq n\leq N$, we have $ \|\nabla u_h^n\|+\sum_{i=1}^k\|\nabla v_{i, h}^n\| \leq Q$ for the AUX* scheme for $\tau$ and $h$ small enough.
\end{theorem}
\begin{proof}
	We take $\chi_2= v_{i,h}^{n}-v_{i,h}^{n-1}$ in the second equation of the AUX* scheme and apply Theorem \ref{orthR} and 
\begin{equation*}
	\begin{aligned}
&\tau\big|\big(\widetilde{\mathcal G}^n v_{i, h}^{n-\frac{1}{2}}, v_{i, h}^{n}-v_{i, h}^{n-1}\big)\big|\leq \tau\|\mathbf{b} \|_{L^\infty}\| \nabla v_{i, h}^{n-\frac{1}{2}}\|\cdot \|v_{i,h}^{n}-v_{i,h}^{n-1}\|+Q\tau\|v_{i,h}^{n}-v_{i,h}^{n-1}\| \\[0.05in]
&\qquad\quad\leq \frac{\tau^2\|\mathbf{b}\|_{L^\infty}^2\gamma}{2}(\|\nabla v_{i,h}^{n-1}\|^2+\|\nabla v_{i,h}^n\|^2)+\frac{3\gamma^{-1}}{4}\|v_{i, h}^{n}-v_{i, h}^{n-1}\|^2+Q\tau^2	
\end{aligned}
\end{equation*}
	 to get
\begin{equation*}
	\begin{aligned}
& \ds 2\gamma^{-1}\|v^{n}_{i, h}-v^{n-1}_{i, h}\|^2 =\tau (v_{i, h}^{n-\frac{1}{2}}, v_{i,h}^{n})(\|\sqrt{\mathbf{a}}\nabla v_{i, h}^{n-1}\|^2-\|\sqrt{\mathbf{a}}\nabla v_{i, h}^{n}\|^2 ) \\[0.05in]
&\quad + 2\tau (v_{i, h}^{n-\frac{1}{2}}, v_{i,h}^{n})\big(\widetilde{\mathcal G}^n v_{i, h}^{n-\frac{1}{2}}, v_{i, h}^{n}-v_{i, h}^{n-1}
\big)\\
&\leq \tau(v_{i, h}^{n-\frac{1}{2}}, v_{i,h}^{n})(\|\sqrt{\mathbf{a}}\nabla v_{i, h}^{n-1}\|^2-\|\sqrt{\mathbf{a}}\nabla v_{i, h}^{n}\|^2 )+\frac{3\gamma^{-1}}{2}(v_{i, h}^{n-\frac{1}{2}}, v_{i,h}^{n})\|v_{i, h}^{n}-v_{i, h}^{n-1}\|^2\\
&\quad+(v_{i, h}^{n-\frac{1}{2}}, v_{i,h}^{n})\Big[\tau^2\|\mathbf{b}\|_{L^\infty}^2\gamma(\|\nabla v_{i,h}^{n-1}\|^2+\|\nabla v_{i,h}^n\|^2)\Big]+Q\tau^2\\
&\leq -\tau(v_{i, h}^{n-\frac{1}{2}}, v_{i,h}^{n})d+ \frac{3\gamma^{-1}}{2}\|v^{n}_{i, h}\!-v^{n-1}_{i, h}\|^2+Q\tau^2,
\end{aligned}
\end{equation*}
that is,
$
\frac{\gamma^{-1}}{2}\|v^{n}_{i, h}-v^{n-1}_{i, h}\|^2 \leq -\tau(v_{i, h}^{n-\frac{1}{2}}, v_{i,h}^{n})d+Q\tau^2,
$
where 
$$d:=\|\sqrt{\mathbf{a}}\nabla v_{i, h}^{n}\|^2-\tau\gamma\|\mathbf{b}\|_{L^\infty}^2\| \nabla v_{i, h}^{n}\|^2-(\|\sqrt{\mathbf{a}}\nabla v_{i, h}^{n-1}\|^2+\tau\gamma\|\mathbf{b}\|_{L^\infty}^2\|\nabla v_{i, h}^{n-1}\|^2).$$
 Then a similar estimate as those around (\ref{jy15})--(\ref{Zjy1}) gives $d\leq Q\tau$, and
 we invoke this, the definition of $d$ and (\ref{zjy15}) to obtain
$$\|\sqrt{\mathbf{a}}\nabla v_{i, h}^{n}\|^2\leq  \|\sqrt{\mathbf{a}}\nabla v_{i, h}^{n-1}\|^2+\frac{\tau\gamma\|\mathbf{b}\|_{L^\infty}^2}{a_0}(\|\sqrt{\mathbf{a}}\nabla v_{i, h}^{n-1}\|^2+\|\sqrt{\mathbf{a}}\nabla  v_{i, h}^{n}\|^2)+ Q\tau. $$
Sum this equation from $n = 1$ to $n^* \leq N$ and apply the Gronwall inequality and \textit{Assumption A} to get 
$a_0\|\nabla v_{i,h}^{n^*}\|^2\leq \|\sqrt{\mathbf{a}}\nabla v_{i,h}^{n^*}\|^2 \leq Q(1+\|\sqrt{\mathbf{a}}\nabla v_{i,h}^{0}\|^2)\leq Q.$
The estimate of $\nabla u_h^{n}$ follows a similar  process around (\ref{Bu}) and we thus omit the proof.
\end{proof}

With the assistance of Theorem \ref{stable2}, we follow the same procedure as Section \ref{er} to derive error estimate.
\begin{theorem}
Under Assumption A,  (\ref{unless2}) for $1\leq n\leq N$ and the regularity condition $u,v_i\in H^1(0,T;H^2)\cap H^2(0,T;L^2)$ for $1\leq i\leq k$ and the time-step condition $\tau=o(h^{d/2})$, the  error estimate $\|u(t_n)-u_h^n\|+\sum_{i=1}^k\|v_i(t_n)-v_{i,h}^n\|\leq Q(\tau+h^2)$  of (\ref{Rapp})  holds for for $1\leq n\leq N$ and $h$ and $\tau$ small enough.
\end{theorem} 

Based on the error estimate, one could follow a similar proof as Corollary \ref{cor1} to conclude that, under similar conditions of Corollary \ref{cor1}, the scheme (\ref{Rapp}) preserves the indices of saddle solutions. 

\section{Numerical experiments}\label{sec7} We present numerical examples to show the performance of the scheme (\ref{app}) (or (\ref{Rapp})) and substantiate the theoretical results. The uniform partitions are used for both the time period $[0,T]$  with $T=5$ and space domain $(0,\pi)^d$ with time step size $\tau=T/N$ and spatial mesh size $h=\pi/M$ for some integers $N$, $M>0$. The errors are measured by $\text{Err}(u):=\max\limits_{1\leq n\leq N}\|u(t_n)-u_h^n\|$ and $\text{Err}(v_i):=\max\limits_{1\leq n\leq N}\|v_i(t_n)-v_{i,h}^n\|.$
 For $d=1$, we define the vector $F(u_h(T))\in\mathbb R^{M-1}$ as $F(u_h(T))_i:=(\mathcal L_{h}u_h(T))(x_i)+ f(u_h(x_i,T))$ for $1\leq i\leq M-1$ where $\mathcal L_{h}$ denotes the approximation of $\mathcal L$ determined from the finite element scheme and $x_i$ denotes the $i$th internal node of the spatial partition. Thus $F(u_h(T))$ is used to measure the degree of $u_h(T)$ in satisfying the equation (and thus in approximating the saddle point). Throughout the numerical experiments, we always set $\beta =\gamma =1$.

\textbf{Example 1: One-dimensional case.} We first consider model (\ref{elliptic}) with $d=1$, $a(x)=1$ and $f(u)=u^4-10u^2$.
We test the performance of the scheme (\ref{app}) with $k=3$, using parameters $\tau=10^{-3}$, $h=\pi/100$, $v_{i,0}=\sqrt{2/\pi}\sin (ix)$ ($1\leq i\leq k $), and $u_0$ in Figure \ref{fig1}(a), which show that the method successfully converges to a saddle point. The $\|F(u_h(T))\|_{l^\infty}$  is $1.41\times 10^{-6}$, indicating that the scheme (\ref{app}) successfully locates saddle point of the semilinear elliptic equation (\ref{elliptic}).

Then we consider model (\ref{RDC}) with $\mathbf{a}(x)=0.02$, $\mathbf{b}(x)=0.02\sin 2x$, $c(x)=0.5$, $ f(u)=-u^2$. We evaluate the performance of the scheme (\ref{Rapp}) with $k=3$ under the same parameters as before with initial condition $u_0$ in Figure \ref{fig1}(b). The $\|F(u_h(T))\|_{l^\infty}$ is $4.08\times 10^{-3}$, which again demonstrates that the method successfully converges to a saddle point of the semilinear advection-reaction-diffusion equation (\ref{RDC}).

\begin{figure}[h]
\vspace{-0.2in}
\centering
\subfloat[$u_0=\sin 4x$]{\label{figa1} 
    \includegraphics[width=2.1in,height=0.9in]{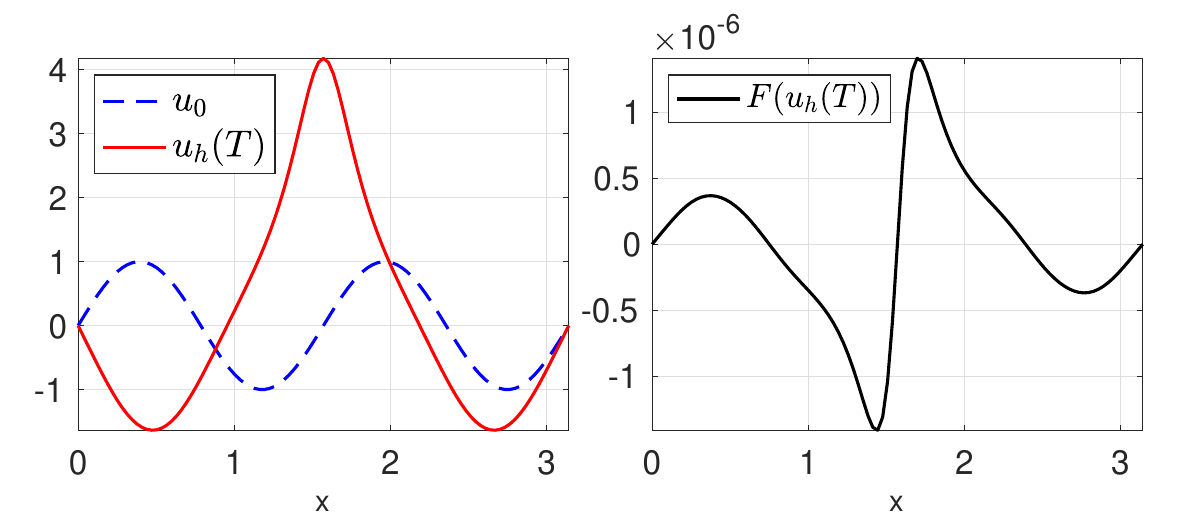}}
\hspace{0.1in}
\subfloat[$u_0=0.1x\sin x$]{\label{figa4}
    \includegraphics[width=2.1in,height=0.9in]{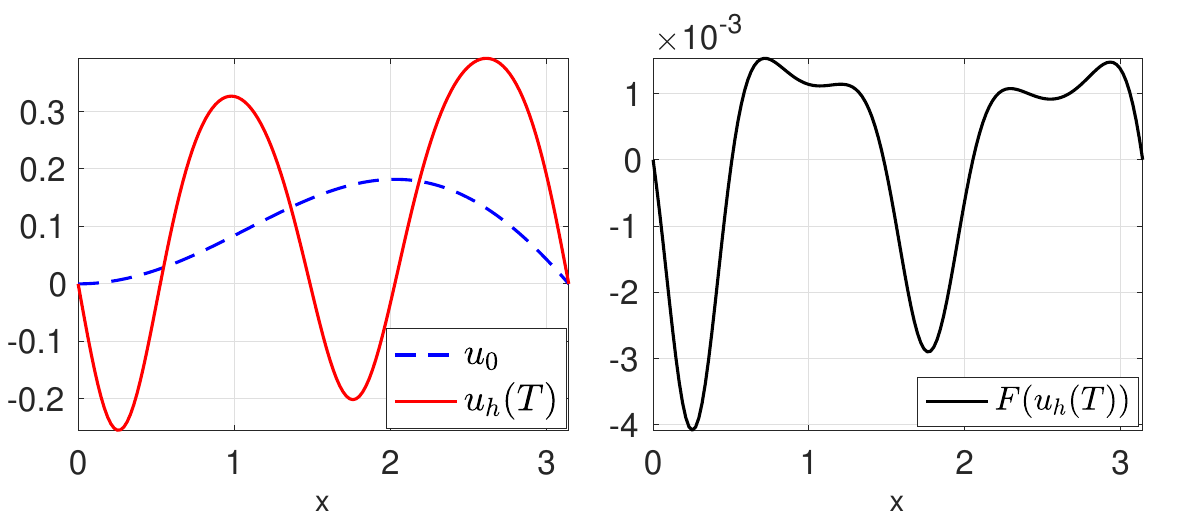}}
    
\caption{Plots of $u_h(T)$ and $F(u_h(T))$ in Example 1.}
\label{fig1}
\end{figure}
\vspace{-0.2in}
To show the index-preservation of the schemes, we compute eigenvalues of the Hessian at $u_h(T)$ for cases (a) and (b) in Figure \ref{fig1} under $\tau=10^{-3}$ and different $h$. It can be observed from Figure \ref{eig} that for $h$ not sufficiently small, the eigenvalues are not exact and their signs are even incorrect, which can not give the correct indices of saddle points. As $h$ becomes sufficiently small, the values of the eigenvalues become stable, leading to the correct indices of saddle points. These observations are consistent with Corollary \ref{cor1} and imply the index-preservation of the proposed schemes. 

\begin{figure}[h]
\vspace{-0.2in}
\centering
\captionsetup[subfloat]{labelformat=empty}
\subfloat[Case (a)]{
    \includegraphics[width=2in,height=1.2in]{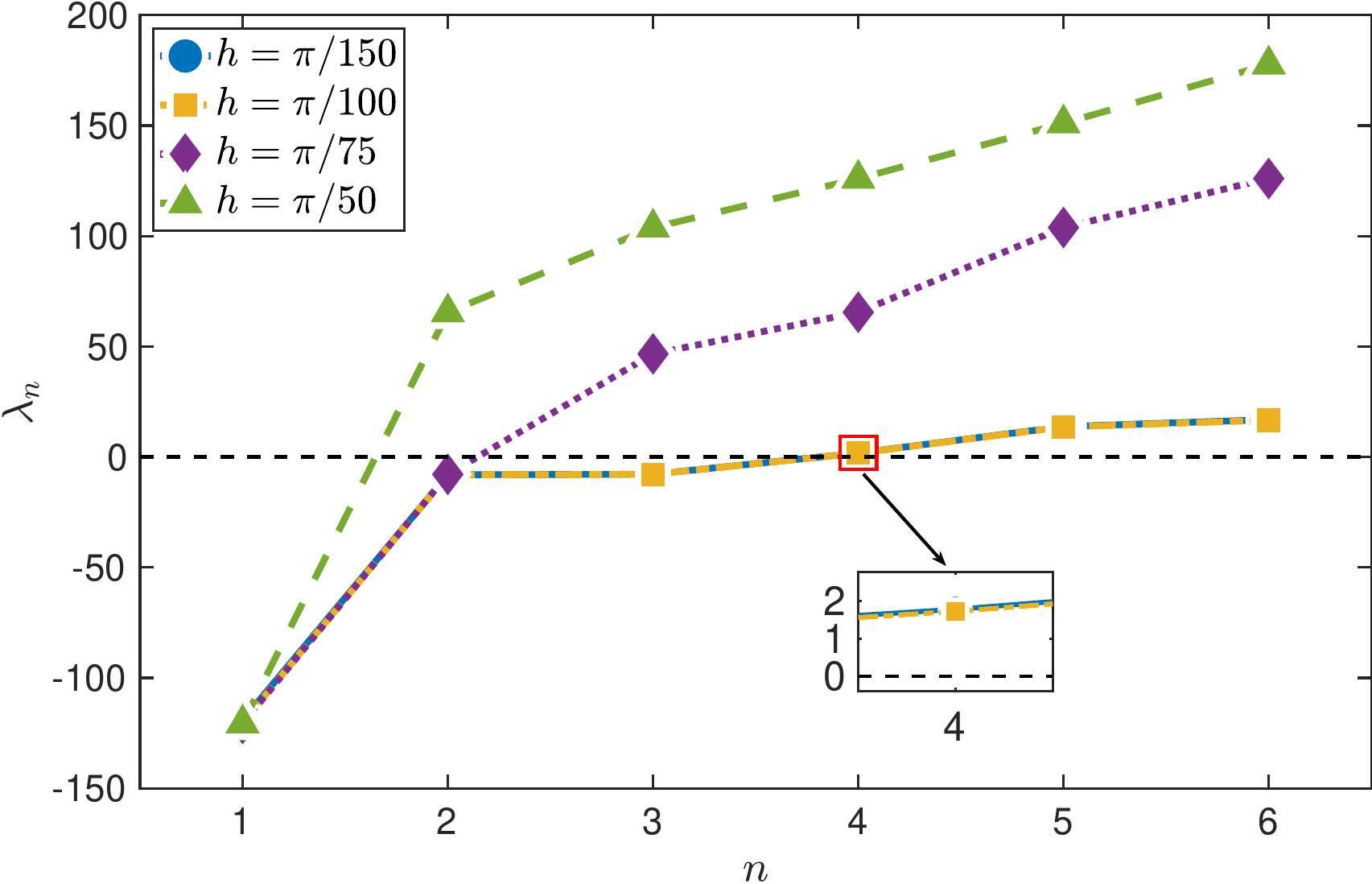}} 
    \hspace{0.1in}
\subfloat[Case (b)]{
    \includegraphics[width=2in,height=1.2in]{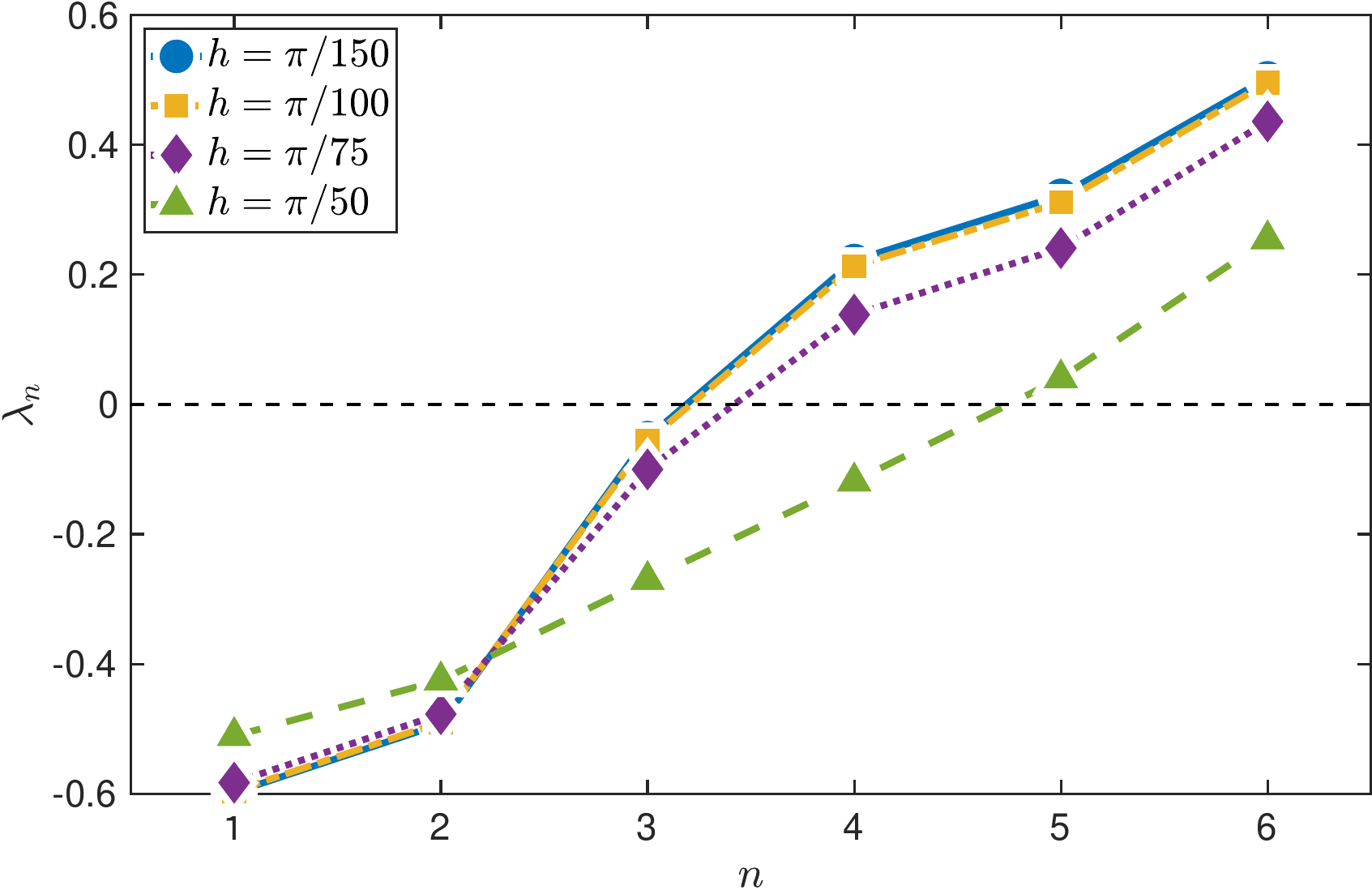}}
\caption{The first six smallest eigenvalues of the Hessian at $u_h(T)$ in Example 1.}
\label{eig}
\end{figure}

\vspace{-0.2in}

\begin{table}[h]
\caption{Accuracy tests for Case (a) in Example 1.}
\vspace{-0.1in}
\centering
\footnotesize{\begin{tabular}{ccccc cccc}
\hline
$\tau\, (h=\pi/1024)$& $\text{Err}(u)$ & rate   & $\text{Err}(v_1)$ & rate  & $\text{Err}(v_2)$ & rate   & $\text{Err}(v_3)$ & rate\\
\hline
$1.6\times 10^{-2}$  & 1.43e-03  &   --   &   2.42e-03 &   -- & 7.53e-03  &  --  &   1.39e-02 &--\\
$8\times 10^{-3}$    & 7.75e-04	 & 0.88 &   1.27e-03 & 0.93   & 4.07e-03  & 0.89 &   7.32e-03 & 0.93\\
$4\times 10^{-3}$    & 3.80e-04  & 1.03 &   6.29e-04 & 1.01   & 2.07e-03  & 0.98 &   3.72e-03 & 0.98\\ 
$2\times 10^{-3}$    & 1.92e-04  & 0.98 &   3.06e-04 & 1.04   & 1.06e-04  & 0.97 &   1.88e-03 & 0.98\\ 
\hline
$h\,(\tau =10^{-4})$ & $\text{Err}(u)$ & rate   & $\text{Err}(v_1)$ & rate  & $\text{Err}(v_2)$ & rate   & $\text{Err}(v_3)$ & rate\\
\hline
$\pi/64$     & 2.12e-02  &  --  &   1.37e-02 & --     & 1.21e-02  & --   &   1.01e-02 & --\\
$\pi/128$    & 5.53e-03	 & 1.94 &   3.59e-03 & 1.93   & 3.12e-03  & 1.96 &   2.58e-03 & 1.97\\
$\pi/256$    & 1.40e-04  & 1.99 &   9.10e-04 & 1.98   & 7.87e-04  & 1.99 &   6.50e-04 & 1.99\\ 
$\pi/512$    & 3.50e-04  & 2.00 &   2.28e-04 & 2.00   & 1.97e-04  & 2.00 &   1.63e-04 & 2.00\\ 
\hline
\end{tabular}}
\label{conv1}
\end{table}

We finally test the convergence rates of the scheme (\ref{app}) under case (a) for illustration, where the reference solution is computed with $\tau=10^{-4}$ and $h=\pi/1024$. 
Table \ref{conv1} indicates first and second order accuracy in time and space, respectively, confirming theoretical predictions.


\textbf{Example 2: Two-dimensional case.} We consider model (\ref{elliptic}) on $\Omega = (0,1)^2$ with $f(u)=u-u^3$ and $\mathbf{a} (x)=\text{diag}\{0.006,0.006\}$, which corresponds to the phase field model. Instead of computing single solutions by scheme (\ref{app}), we combine this scheme under $\tau=10^{-3}$ and $h=1/64$ with the solution landscape algorithm \cite{YinSCM} to systematically find multiple solutions. The solution landscape is depicted in Figure \ref{fig2}, where each image represents a solution to (\ref{elliptic}).

Compared with the solution landscape of the same model with periodic boundary condition (cf. \cite[Figure 3(d)]{YinSCM}),  the Dirichlet boundary condition  results in more complex solution patterns. Specifically, the solution landscape in \cite[Figure 3(d)]{YinSCM} contains 10 different solutions (for rotationally equivalent solutions or solutions that differ by a sign we only count in one of them), while the solution landscape in Figure \ref{fig2}  contains 16 different solutions. The results substantiate the effectiveness of the scheme (\ref{app}) in locating multiple solutions of semilinear elliptic problems.

\vspace{-0.1in}

\begin{figure}[h]
\centering 
\includegraphics[width=4in,height=1.3in]{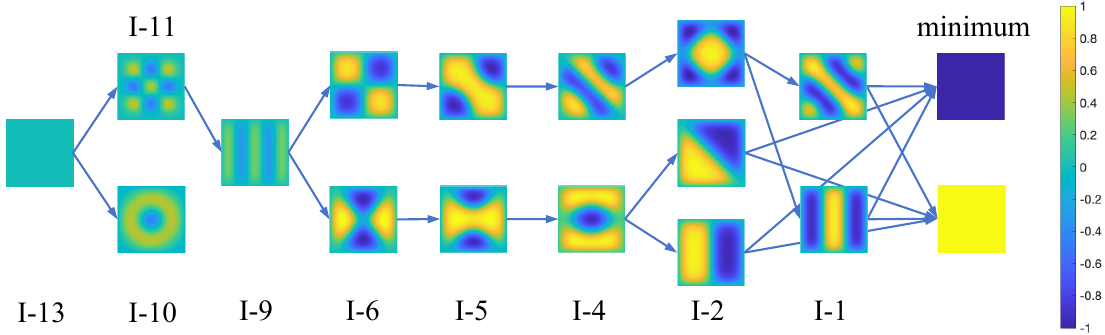} 
\caption{Solution landscape for (\ref{elliptic}) in Example 2.}
\label{fig2} 
\end{figure}

\vspace{-0.34in}

\section{Concluding remarks}
This paper presents a rigorous numerical framework for spatiotemporal HiSD, featuring a retraction-free orthonormality-preserving scheme for semilinear elliptic problems. Theoretical analysis and numerical experiments have demonstrated the scheme's stability and accuracy. 
This work lays the foundation for extending HiSD to more complex infinite-dimensional systems, including the Gross-Pitaevskii equation \cite{BaoDu}, the Kohn--Sham density functional model \cite{Dai} and the reaction-diffusion system \cite{Hao2}. 
Additionally, the flexibility of the framework allows for the integration of different spatial discretization methods to accommodate diverse PDE structures. We plan to explore these promising directions to further broaden the scope of multiple solution computations in future research.

\section*{Appendix: Auxiliary estimates}
We present truncation errors  appeared in Section \ref{er} and derive their bounds. The truncation errors  $\mu_1^n$--$\mu_3^n$  in (\ref{error}) are given as 
\begin{equation*}
\left\{
\begin{array}{l}
\ds \mu_1^n:= 2\sum_{i=1}^k(v^{n-1}_{i, h}, \chi_1)(\mathbf{a}\nabla v^{n-1}_{i, h}, \nabla u^{n}_h)- 2\sum_{i=1}^k\big(v_i(t_{n}), \chi_1)(\mathbf{a}\nabla v_i(t_{n}), \nabla u(t_{n})\big), \\[0.1in]
\ds \mu_2^n: =2\sum_{i=1}^{k}(v^{n-1}_{i, h},\chi_1)(\tilde f(u^{n-1}_h), v^{n-1}_{i, h}) - 2\sum_{i=1}^k(v_i(t_{n}),\chi_1)\big(\tilde f(u(t_{n})), v_i(t_{n})\big)
\end{array}
\right.
\end{equation*}
and  $\mu_3^n: =\big(\tilde f(u(t_n))-\tilde f(u^{n-1}_h), \chi_1\big)$. Following similar derivations in the first part of  \cite[Section 4.4]{ZhaZheZhu}, the $\mu_{1}^n$--$\mu_{3}^n$ with $\chi_1= \xi_{u}^{n}$ can be bounded as
$
\sum_{j=1}^3|\mu_j^n| \leq Q\big(\sum_{i=1}^k\|e_{v_i}^{n-1}\|+\|e_u^{n-1}\|+\|\sqrt{\mathbf{a}} \nabla \xi_u^{n}\|+\tau\big) \|\xi_u^n\|.
$

 The truncation errors   $\nu_1^n$--$\nu_4^n$ in (\ref{error}) are given as 
\begin{equation*}
\left\{
\begin{array}{l}
\ds \nu_{i,1}^n := (v_i(t_n), v_i(t_n))(\widetilde{f'}(u(t_{n}))v_i(t_{n}),\chi_2)-(v_{i, h}^{n-\frac{1}{2}}, v_{i, h}^{n-\frac{1}{2}})(\widetilde{f'}(u_h^n)v_{i, h}^{n-\frac{1}{2}}, \chi_2)\\[0.1in]
 \ds\qquad -\frac{1}{2}(\mathbf{a}\nabla v_i(t_{n})-\mathbf{a}\nabla v_i(t_{n-1}), \nabla \chi_2)-O(\tau)(\mathbf{a}\nabla v^{n-\frac{1}{2}}_{i, h}, \nabla \chi_2), \\[0.2in]
\ds \nu_{i,2}^n: = (v_i(t_{n}), \chi_2)\big[(\mathbf{a}\nabla v_i(t_{n}),\nabla v_i(t_{n}))-(\widetilde{f'}(u(t_{n}))v_i(t_{n}), v_i(t_{n}))\big]\\[0.15in]
\qquad -(v_{i, h}^{n-\frac{1}{2}}, \chi_2)\big[(\mathbf{a}\nabla v_{i, h}^{n-\frac{1}{2}}, \nabla v_{i, h}^{n-\frac{1}{2}})-(\widetilde{f'}(u_h^n)v_{i, h}^{n-\frac{1}{2}}, v_{i, h}^{n-\frac{1}{2}})\big],\\[0.2in]
\ds \nu_{i,3}^n: = -\sum_{l< i}(v_{i, h}^{n-\frac{1}{2}}, v_{i, h}^{n-\frac{1}{2}})(v_{l,h}^{n}, \chi_2)\big[(\mathbf{a}\nabla v_{i, h}^{n-\frac{1}{2}}, \nabla v_{l, h}^{n})-(\widetilde{f'}(u_h^n)v_{i, h}^{n-\frac{1}{2}}, v_{l, h}^{n})\big]\\[0.15in]
\ds +\sum_{l<i}(v_i(t_n), v_i(t_n))(v_l(t_{n}), \chi_2)\big[(\mathbf{a}\nabla v_i(t_{n}),\nabla v_l(t_{n}))-(\widetilde{f'}(u(t_{n}))v_i(t_{n}), v_l(t_{n}))\big],\\[0.2in]
\ds \nu_{i,4}^n:=-\sum_{l> i}(v_{i, h}^{n-\frac{1}{2}}, v_{i, h}^{n-\frac{1}{2}})(v_{l,h}^{n-1}, \chi_2)\big[(\mathbf{a}\nabla v_{i, h}^{n-\frac{1}{2}}, \nabla v_{l, h}^{n-1})-(\widetilde{f'}(u_h^n)v_{i, h}^{n-\frac{1}{2}}, v_{l, h}^{n-1})\big]\\[0.15in]
 \ds + \sum_{l>i}(v_i(t_n), v_i(t_n))(v_l(t_{n}), \chi_2)\big[(\mathbf{a}\nabla v_i(t_{n}),\nabla v_l(t_{n}))-(\widetilde{f'}(u(t_{n}))v_i(t_{n}), v_l(t_{n}))\big].
\end{array}
\right.
\end{equation*}


We then bound $\nu_{i,1}^n$--$\nu_{i,4}^n$ with $\chi_2= \xi_{v_i}^{n-\frac{1}{2}}$. By (\ref{zjy20}), direct splittings and calculations give the estimate of $\nu_{i,1}^n$
\begin{align*}
| \nu_{i,1}^n | & = \big|(\widetilde{f'}(u(t_{n}))v_i(t_{n}),\xi_{v_i}^{n-\frac{1}{2}})-(1+O(\tau))(\widetilde{f'}(u_h^n)v_{i, h}^{n-\frac{1}{2}}, \xi_{v_i}^{n-\frac{1}{2}})\\
 &\qquad -\frac{1}{2}(\mathbf{a} \nabla v_i(t_{n})-\mathbf{a}\nabla v_i(t_{n-1}), \nabla \xi_{v_i}^{n-\frac{1}{2}})-O(\tau)(\mathbf{a}\nabla v^{n-\frac{1}{2}}_{i, h}, \nabla \xi_{v_i}^{n-\frac{1}{2}}) \big| \\
 & = \big| \big((\widetilde{f'}(u(t_{n}))-\widetilde{f'}(u_h^n))v_i(t_{n}),\xi_{v_i}^{n-\frac{1}{2}}\big) + \frac{1}{2}\big (\widetilde{f'}(u_h^n)(v_i(t_{n})+ v_i(t_{n-1})),\xi_{v_i}^{n-\frac{1}{2}}\big) \\
 &\quad  -(1+O(\tau))(\widetilde{f'}(u_h^n)v_{i, h}^{n-\frac{1}{2}}, \xi_{v_i}^{n-\frac{1}{2}})+ \frac{1}{2}(\widetilde{f'}(u_h^n)(v_i(t_{n})- v_i(t_{n-1})),\xi_{v_i}^{n-\frac{1}{2}}\big)\\
 &\quad -\frac{1}{2}(\mathbf{a}\nabla v_i(t_{n})-\mathbf{a}\nabla v_i(t_{n-1}), \nabla \xi_{v_i}^{n-\frac{1}{2}})-O(\tau)(\mathbf{a}\nabla v^{n-\frac{1}{2}}_{i, h}, \nabla \xi_{v_i}^{n-\frac{1}{2}}) \big|\\
 &\leq Q \|\xi_{v_i}^{n-\frac{1}{2}}\|\cdot (\|e_u^n\|+  \|e_{v_i}^n\|+  \|e_{v_i}^{n-1}\|) +Q\tau\|\xi_{v_i}^{n-\frac{1}{2}}\|+Q\tau \|\sqrt{\mathbf{a}}\nabla \xi_{v_i}^{n-\frac{1}{2}}\|,
	\end{align*}
where we have used $\|\sqrt{\mathbf{a}}(\nabla v_i(t_{n})-\nabla v_i(t_{n-1}))\|\leq Q\tau\|(\nabla v_i)_t\|_{C([0,T];L^2)}$.
We then split $\nu_{i,2}^n$ as $ A_1+A_2+A_3$ where
	\begin{align*}
A_1 & = \frac{1}{2} \Big\{(v_i(t_{n}), \xi_{v_i}^{n-\frac{1}{2}})\big[(\mathbf{a}\nabla v_i(t_{n}),\nabla v_i(t_{n}))-(\widetilde{f'}(u(t_{n}))v_i(t_{n}), v_i(t_{n}))\big]\\
&\qquad -(v_{i, h}^{n}, \xi_{v_i}^{n-\frac{1}{2}})\big[(\mathbf{a}\nabla v_{i, h}^{n-\frac{1}{2}}, \nabla v_{i, h}^{n-\frac{1}{2}})-(\widetilde{f'}(u_h^n)v_{i, h}^{n-\frac{1}{2}}, v_{i, h}^{n-\frac{1}{2}})\big]\Big\},\\
A_2&=\frac{1}{2}\Big\{(v_i(t_{n-1}), \xi_{v_i}^{n-\frac{1}{2}})\big[(\mathbf{a}\nabla v_i(t_{n}),\nabla v_i(t_{n}))-(\widetilde{f'}(u(t_{n}))v_i(t_{n}), v_i(t_{n}))\big]\\
&\qquad -(v_{i, h}^{n-1}, \xi_{v_i}^{n-\frac{1}{2}})\big[(\mathbf{a}\nabla v_{i, h}^{n-\frac{1}{2}}, \nabla v_{i, h}^{n-\frac{1}{2}})-(\widetilde{f'}(u_h^n)v_{i, h}^{n-\frac{1}{2}}, v_{i, h}^{n-\frac{1}{2}})\big]\Big\},\\
A_3&= \frac{1}{2}(v_i(t_{n})-v_i(t_{n-1}), \xi_{v_i}^{n-\frac{1}{2}})\big[(\mathbf{a}\nabla v_i(t_{n}),\nabla v_i(t_{n}))-(\widetilde{f'}(u(t_{n}))v_i(t_{n}), v_i(t_{n}))\big].
\end{align*}
We bound $A_1$ for demonstration, and $A_2$ and $A_3$ could be bounded similarly and thus we omit the details. For the sake of analysis, we further rewrite $A_1$ as 
\begin{align*}
A_1 &= \frac{1}{4}\Big\{(v_i(t_{n}), \xi_{v_i}^{n-\frac{1}{2}})\big[(\mathbf{a}\nabla v_i(t_{n}),\nabla v_i(t_{n}))-(\widetilde{f'}(u(t_{n}))v_i(t_{n}), v_i(t_{n}))\big]\\
&\qquad -(v_{i, h}^{n}, \xi_{v_i}^{n-\frac{1}{2}})\big[(\mathbf{a}\nabla v_{i, h}^{n}, \nabla v_{i, h}^{n-\frac{1}{2}})-(\widetilde{f'}(u_h^n)v_{i, h}^{n}, v_{i, h}^{n-\frac{1}{2}})\big]\Big\}\\		
&\quad +\frac{1}{4}\Big\{(v_i(t_{n}), \xi_{v_i}^{n-\frac{1}{2}})\big[(\mathbf{a}\nabla v_i(t_{n-1}),\nabla v_i(t_{n}))-(\widetilde{f'}(u(t_{n}))v_i(t_{n-1}), v_i(t_{n}))\big]\\
&\qquad -(v_{i, h}^{n}, \xi_{v_i}^{n-\frac{1}{2}})\big[(\mathbf{a}\nabla v_{i, h}^{n-1}, \nabla v_{i, h}^{n-\frac{1}{2}})-(\widetilde{f'}(u_h^n)v_{i, h}^{n-1}, v_{i, h}^{n-\frac{1}{2}})\big]\Big\}\\
&\quad + \frac{1}{4}\Big\{(v_i(t_{n}), \xi_{v_i}^{n-\frac{1}{2}})(\mathbf{a}(\nabla v_i(t_{n})-\nabla v_i(t_{n-1})),\nabla v_i(t_{n}))\\
&\qquad -(v_i(t_{n}), \xi_{v_i}^{n-\frac{1}{2}})(\widetilde{f'}(u(t_{n}))(v_i(t_{n})-v_i(t_{n-1}), v_i(t_{n}))\Big\}=:B_1+B_2+B_3.	
\end{align*}
We first bound $B_1$ as follows 
	\begin{align}
|B_1| & =\frac{1}{8} \big|(v_i(t_{n}), \xi_{v_i}^{n-\frac{1}{2}})(\mathbf{a}\nabla v_i(t_{n}),\nabla v_i(t_{n})+\nabla v_i(t_{n-1}))\nonumber\\
& \qquad -(v_{i, h}^{n}, \xi_{v_i}^{n-\frac{1}{2}})(\mathbf{a}\nabla v_{i, h}^{n}, \nabla v_{i, h}^{n}+\nabla v_{i, h}^{n-1})\nonumber\\
& \qquad -(v_i(t_{n}), \xi_{v_i}^{n-\frac{1}{2}})(\widetilde{f'}(u(t_{n}))v_i(t_{n}), v_i(t_{n})+v_i(t_{n-1}))	\nonumber\\
&\qquad + (v_{i, h}^{n}, \xi_{v_i}^{n-\frac{1}{2}})(\widetilde{f'}(u_h^n)v_{i, h}^{n}, v_{i, h}^{n}+v_{i, h}^{n-1})\label{short1}\\
&\qquad +(v_i(t_{n}), \xi_{v_i}^{n-\frac{1}{2}})\big[(\mathbf{a}\nabla v_i(t_{n}),\nabla v_i(t_{n})-\nabla v_i(t_{n-1}))\nonumber\\
&\qquad -(v_i(t_{n}), \xi_{v_i}^{n-\frac{1}{2}})(\widetilde{f'}(u(t_{n}))v_i(t_{n}), v_i(t_{n})-v_i(t_{n-1}))\big|\nonumber\\
&\leq Q\big(\|e_{v_i}^{n}\|+ \|e_{v_i}^{n-1}\|+  \|\sqrt{\mathbf{a}}\nabla \xi_{v_i}^{n-\frac{1}{2}}\| + \|e_{u}^{n}\|+\tau\big)\|\xi_{v_i}^{n-\frac{1}{2}}\|,  \nonumber
\end{align}
where we have used the integration by parts 
\begin{equation*}
	\begin{aligned}
	&(\mathbf{a}\nabla v_i(t_{n}),\nabla v_i(t_{n})+\nabla v_i(t_{n-1}))-(\mathbf{a}\nabla v_{i, h}^{n}, \nabla v_{i, h}^{n}+\nabla v_{i, h}^{n-1})\\
	&\quad = (\nabla v_i(t_{n})-\nabla v_{i, h}^{n}, \mathbf{a}\nabla v_i(t_{n})+\mathbf{a}\nabla v_i(t_{n-1}))\\
	&\qquad +(\mathbf{a}\nabla v_{i, h}^{n}, \nabla v_i(t_{n})-\nabla v_{i, h}^{n}+\nabla v_i(t_{n-1})-\nabla v_{i, h}^{n-1})\\
	&\quad =- ( e_{v_i}^n, \nabla\cdot (\mathbf{a}\nabla v_i(t_{n}))+\nabla\cdot (\mathbf{a}\nabla v_i(t_{n-1})))+ (\mathbf{a}\nabla v_{i, h}^{n}, \nabla \xi_{v_i}^n+\nabla \xi_{v_i}^{n-1}).
	\end{aligned}
\end{equation*}
We then similarly bound $B_2$ and $B_3$ by the right-hand side of (\ref{short1}).
We summarize the above estimates to  bound  $\nu^n_{i,2}$ by the right-hand side of (\ref{short1}).

By (\ref{zjy20}), the $\nu_{i,3}^n$ could be split as 
\begin{equation*}
	\begin{aligned}
\ds  \nu_{i,3}^n  
& = \sum_{l<i} (v_l(t_{n}), \xi_{v_i}^{n-\frac{1}{2}})\big[(\mathbf{a}\nabla v_i(t_{n}),\nabla v_l(t_{n}))-(\widetilde{f'}(u(t_{n}))v_i(t_{n}), v_l(t_{n}))\big]\\
\ds & \qquad -\sum_{l<i}(v_{l,h}^{n}, \xi_{v_i}^{n-\frac{1}{2}})\big[(\mathbf{a}\nabla v_{i, h}^{n-\frac{1}{2}}, \nabla v_{l, h}^{n})-(\widetilde{f'}(u_h^n)v_{i, h}^{n-\frac{1}{2}}, v_{l, h}^{n})\big] \\
&\qquad +O(\tau)\sum_{l<i}(v_{l,h}^{n}, \xi_{v_i}^{n-\frac{1}{2}})\big[(\mathbf{a}\nabla v_{i, h}^{n-\frac{1}{2}}\nabla v_{i, h}^{n-\frac{1}{2}}, \nabla v_{l, h}^{n})-(\widetilde{f'}(u_h^n)v_{i, h}^{n-\frac{1}{2}}, v_{l, h}^{n})\big],
\end{aligned}
\end{equation*}
and we could bound $\nu_{i,3}^n$ in a similar manner as the estimate of  $\nu_{i,2}^n$ to get
\begin{equation*}
|\nu_{i,3}^n |  \leq Q \Big(\sum_{l\leq i}\|e_{v_l}^{n} \| + \|e_{v_i}^{n-1} \|+ \|e_u^n\|  +  \|\sqrt{\mathbf{a}}\nabla \xi_{v_i}^{n-\frac{1}{2}} \|+\tau\Big)\|\xi_{v_i}^{n-\frac{1}{2}}\|.
\end{equation*}
By similar splittings, the $\nu_{i,4}^n$ could be bounded by the right-hand side of the above equation with the first two terms replaced by  $Q \sum_{l \geq  i} \|e_{v_l}^{n-1}\|\cdot \|\xi_{v_i}^{n-\frac{1}{2}}\|+Q \|e_{v_i}^{n} \|\cdot \|\xi_{v_i}^{n-\frac{1}{2}}\|$.

\end{document}